\documentclass[nonblindrev]{informs3}

\OneAndAHalfSpacedXI

\graphicspath{{img/}}
\usepackage{graphicx}
\usepackage[shortlabels]{enumitem}
\usepackage{mathtools}

\usepackage{amsfonts}
\usepackage{caption}
\usepackage{subcaption}

\usepackage{algorithm}
\usepackage{algorithmicx}
\usepackage[noend]{algpseudocode}
\usepackage{multirow}

\usepackage{natbib}
 \bibpunct[, ]{(}{)}{,}{a}{}{,}%
 \def\bibfont{\small}%

\usepackage{graphicx}

\usepackage{amsmath,amssymb}
\usepackage{pgfplots}
\usepackage{bbm}

\newtheorem{prop}{Proposition}
\newtheorem{thm}{Theorem}
\newtheorem{lem}{Lemma}
\newtheorem{cor}{Corollary}
\newtheorem{dfn}{Definition}

\DeclareMathOperator*{\lexmin}{lex-min}
\DeclareMathOperator*{\arglexmin}{arglex-min}
\TheoremsNumberedThrough 
\ECRepeatTheorems

\EquationsNumberedThrough 

\MANUSCRIPTNO{123456}

\begin{document}

\RUNAUTHOR{Basciftci and Van Hentenryck}

\RUNTITLE{Capturing Travel Mode Adoption in ODMTS}


\TITLE{Capturing Travel Mode Adoption \\ in Designing On-demand Multimodal Transit Systems}

\ARTICLEAUTHORS{
  \AUTHOR{Beste Basciftci} \AFF{Department of Business Analytics, University of Iowa, Iowa City, Iowa, 52242, USA, \EMAIL{beste-basciftci@uiowa.edu}}
  \AUTHOR{Pascal Van Hentenryck} \AFF{Georgia Institute of Technology, Atlanta, Georgia 30332, USA, \EMAIL{pvh@isye.gatech.edu}}
} 

\ABSTRACT{

This paper studies how to integrate rider mode preferences into the
design of On-Demand Multimodal Transit Systems (ODMTS). It is
motivated by a common worry in transit agencies that the ODMTS may be
poorly designed if the latent demand, i.e., new riders adopting the
system, is not captured. The paper proposes a bilevel optimization
model to address this challenge, in which the leader problem
determines the ODMTS design, and the follower problems identify the
most cost efficient and convenient route for riders under the chosen
design. The leader model contains a choice model for every potential
rider that determines whether the rider adopts the ODMTS given her
proposed route. To solve the bilevel optimization model, the paper
proposes an exact decomposition method that includes Benders optimal
cuts and nogood cuts to ensure the consistency of the rider choices in
the leader and follower problems. Moreover, to improve computational
efficiency, the paper proposes upper \textcolor{black}{and lower}
bounds on trip durations for the follower problems, valid inequalities
that strenghten the nogood cuts, \textcolor{black}{and approaches to
  reduce the problem size with problem-specific preprocessing
  techniques}.

The proposed method is validated using an extensive computational
study on a real data set from AAATA, the transit agency for the
broader Ann Arbor and Ypsilanti region in Michigan. The study
considers the impact of a number of factors, including the price of
on-demand shuttles, the number of hubs, and \textcolor{black}{access to transit systems} criteria. The designed ODMTS feature high adoption rates and
significantly shorter trip durations compared to the existing transit
system and highlight the benefits in \textcolor{black}{ensuring access} for low-income
riders. Finally, the computational study demonstrates the efficiency
of the decomposition method for the case study and the benefits of
computational enhancements that improve the baseline method by several
orders of magnitude.
}%



\maketitle

%

\section{Introduction}

This paper considers On-Demand Multimodal Transit Systems (ODMTS)
\citep{Arthur2019,ISE2019}, a new type of transit systems that combine
on-demand shuttles with fixed routes served by buses or rail. ODMTS
are organized around a number of hubs, on-demand shuttles serve local
demand and act as feeders to and from the hubs, and fixed routes
provide high-frequency service between hubs. By dispatching in real
time on-demand shuttles to pick up riders at their origins and drop
them off at their destinations, ODMTS are ``door-to-door'' and address
the first/last mile problem that plagues most of the transit systems.
Moreover, ODMTS address congestion and economy of scale by providing
high-frequency services along high-density corridors.
{\color{black}Figure \ref{sampleODMTSFigure} presents a sample ODMTS with buses between hubs along with the on-demand shuttles that can serve these hubs.}
They have been
shown to bring substantial convenience and cost benefits in simulation
and pilot studies in the city of Canberra, Australia
\citep{Arthur2019}, the transit system of the University of Michigan
\citep{ISE2019}, the Ann-Arbor/Ypsilanti region in Michigan
\citep{Basciftci2020}, and the city of Atlanta
\citep{TransferGraphs2020}. ODMTS differ from micro-mobility in that
they are designed and operated holistically. The ODMTS design thus
becomes a variant of the hub-arc location problem
\citep{Campbell2005b,Campbell2005a}: It is an optimization model that
decides which bus/rail lines to open in order to maximize convenience
and minimize costs \citep{Arthur2019}. This optimization model uses, as
input, the current demand, i.e., the set of origin-destination pairs
in the existing transit system.

This paper is motivated by a significant worry of transit agencies:
{\em the need to capture latent demand in the design of ODMTS}. This
concern, which recognizes the complex interplay between transit
agencies and riders \citep{Cancela2015}, was also raised by
\cite{Campbell2019}: they articulated the potential of (1) leveraging
data analytics within the planning process and (2) proposing transit
systems that encourage riders to switch transportation modes. As a
consequence, rider preferences and the induced mode choices should be
significant factors in the design of transit systems
\citep{Laporte2007Book}. Yet, many transit agencies only consider
existing riders when redesigning their network. But, as convenience
improves, more riders may decide to switch modes and adopt the transit
system instead of traveling with their personal vehicles. By ignoring
this latent demand, the transit system may be designed suboptimally,
resulting in higher costs or poor quality of
service. \cite{Basciftci2020} illustrated these points by comparing
the designs of ODMTS that differ by whether they capture latent
demand. The results highlighted the significant cost increase when
latent demand is not considered as the design under-invested in fixed
routes and over-utilized on-demand shuttles. Note also that
\cite{Agatz2020Networks} highlighted the integration of stakeholder
behavior in optimization models as a fundamental theme to address
grand challenges in the next generation of transportation systems.

\begin{figure}[h]
\centering
\includegraphics[width=0.85\textwidth]{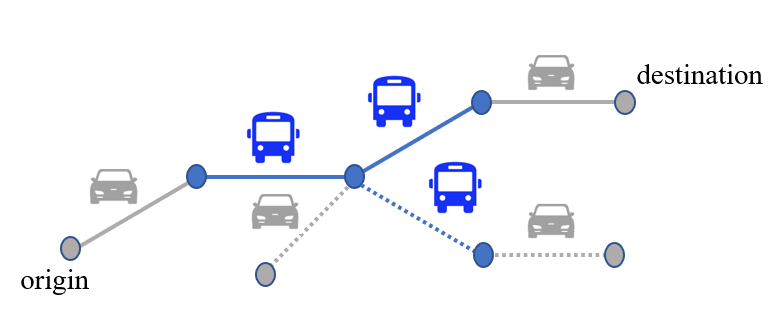}
\caption{\color{black} Illustration of the ODMTS with buses and on-demand shuttles. Hubs are denoted in blue color. Origin and destination stops are denoted in gray color. Solid lines represent the route of a rider from her origin to destination. Dashed lines represent potential bus legs between hubs or on-demand shuttles to/from the hubs.}
\label{sampleODMTSFigure}
\end{figure}

{\color{black} Before presenting the design framework, it is useful to
  review how an ODMTS is used in practice. When a user requests a
  ride, she is presented with the route from origin to destination
  that jointly optimizes system cost and user convenience. The user
  then decides whether she takes the ride or uses a different
  transportation mode.\footnote{Note that maximizing convenience only
    would always result in a direct shuttle trip, defeating the
    multi-modal nature of the ODMTS. Minimizing costs only would often
    result in the user rejecting the ride.} It is thus important to
  realize that users do not choose their routes in the ODMTS: they are
  presented with routes in their mobile applications and decide
  whether to take them. If they could choose the routes, they would
  select a direct shuttle trip.}

The key contribution of this paper is to propose {\em a general
  framework to design an ODMTS based on both existing and latent
  demands}. The framework assumes that the mode preference of a rider
is expressed through a choice model which, given a route in the ODMTS,
determines whether the rider adopts the ODMTS or continues to use her
personal vehicle. {\color{black} The network design problem is then
  formulated as a {\em bilevel optimization} model which can be
  informally understood as follows. There is a subproblem associated
  with each trip by a rider: given a network design (i.e., a choice of
  bus routes to open), this subproblem chooses the route from the trip
  origin to the trip destination that optimizes a weighted combination
  of system cost and user convenience. The master problem chooses a
  network design, obtains the routes of each pair (trip,rider), and
  determines whether the riders take the proposed rides based in their
  choice models. The master problem optimizes an objective function
  that consists of several components: (1) the fixed cost of opening
  bus routes; (2) the cost and convenience of the trips accepted by
  the riders; and (3) the revenue of all adopted trips.} The {\em
  bilevel optimization} model is solved using an exact decomposition
method: it uses traditional Benders optimality cuts and nogood cuts,
which are strengthened by valid inequalities exploiting the network
structure. The approach is validated on a real case study.

The contributions of the paper can be summarized as follows:

\begin{enumerate}
\item The paper presents a bilevel optimization approach to the design
  of ODMTS under rider adoption constraints. The bilevel optimization
  problem consists of (i) a leader problem that determines the transit
  network design and takes into account rider preferences as well as
  revenues and costs from adopting riders; (ii) follower problems
  identify the most cost-efficient and convenient route for riders.
  The personalized choice models are integrated into the leader
  problem to represent the interplay between the transit agency and
  rider preferences. Since the model assumes a fixed cost for riding
  the transit system, the choice models capture the desired
  convenience of the trips. 

\item The paper proposes an exact decomposition method for the bilevel
  optimization model. The method combines a Benders decomposition
  approach with combinatorial cuts that ensure the consistency between
  rider choices and the leader decisions. Furthermore, the paper
  presents valid inequalities that significantly strengthen these
  combinatorial cuts, as well as preprocessing steps that reduce the
  problem dimensionality. These enhancements produce orders of
  magnitude improvements in computation times.

\item The paper validates the approach using a comprehensive case
  study that considers the transit agency of the broad Ann
  Arbor/Ypsilanti region in Michigan. The case study demonstrates the
  benefits of the proposed approach from adoption, convenience, cost
  and \textcolor{black}{access to transit systems} perspectives. The results highlight that the ODMTS
  decreases trip durations by up to 53\% compared to the existing
  system, induces high adoption rates for the latent demand, and operates
  well inside the budget of the transit agency. 
\end{enumerate}

The rest of the paper is organized as follows. Section
\ref{sec:litReview} reviews the relevant literature. Section
\ref{sec:BilevelOptFormulation} presents the problem setting and the
resulting bilevel ODMTS design problem with latent demand and rider
choices. Section \ref{SectionAnalyticalResultsonTripDurations}
proposes theoretical results on trip durations in ODMTS.  Section
\ref{sec:solutionMethodologySection} presents an exact decomposition
algorithm and derives valid inequalities and problem-specific
enhancements. Section \ref{sec:ComputationalStudy} demonstrates the
performance of the proposed approach in the case study. Section
\ref{sec:Conclusion} concludes the paper with final remarks.

\section{Related Literature}
\label{sec:litReview}

The design of transit networks organized around hubs is an emerging
research area, with the goal of ensuring reliable service and
economies of scale \citep{Farahani2013}.
\citet{Campbell2005b,Campbell2005a} introduce a variant of this
problem, the hub-arc location problem, to select the set of arcs to
open between hubs while optimizing the flow with minimum
cost. \citet{Alumur2012} consider multimodal hub location and hub
network design problem by taking into account both cost and
convenience aspects in satisfying demand. \citet{Arthur2019} examine
this problem in the context of ODMTS, pioneering on-demand shuttles to
serve all or parts of the trips, and allowing routes that do not
necessarily involve arcs between hubs. The goal is to obtain a transit
network design that minimizes the cost and duration of the overall
trips. In these studies, user behaviour is not explicitly captured
within the transit network design process; instead the objective
function minimizes a weighted combination of the system cost and the
travel times of the trips for existing riders of the transit system.
\textcolor{black}{Recently, \citet{Steiner2020} studied the design of
  an integrated public bus system with on-demand services.  The paper
  points out the importance of optimizing over a mode-choice model for
  each origin-destination pair for determining rider preferences, and
  it mentions the resulting modeling and computational
  complexities. But the paper does not include mode-choice models:
  instead the formulation precomputes the induced demand based on the
  zones where on-demand service are provided.}

Capturing information about rider routes into transit network design
is a critical component of ensuring accessible public transit systems
\citep{Schobel2012}. \citet{Guan2006} model a joint line planning and
passenger assignment problem as a single-level mixed integer program,
where riders select their routes during network design and the route
durations are part of the objective function along with the costs of
the transit network.  \citet{Borndorfer2007} study this line planning
problem under these two competing objectives by utilizing a
column-generation approach as its solution methodology.
\citet{Schobel2006} consider identifying the routes that minimize the
overall travel time of the riders under a budget constraint on the
transit network design.

Another relevant line of research involving transit network design
problems focuses on maximizing population coverage by examining
population in the neighborhood of the potential stations
\citep{Changshan2005, Matisziw2006, Curtin2011}. In these settings,
travel costs can be jointly optimized with the maximization of
ridership capture \citep{Jarpa2013}. \citet{Marin2009} integrate user
behavior into this planning problem by representing the choices of the
riders according to the network design and the cost of the resulting
trip in comparison to their current mode of travel,
\textcolor{black}{and \citet{Marin2009_UrbanRapidTransit} provide an
  algorithm based on Benders decomposition for its solution}.
\citet{Laporte2011} extend this problem under the possibility of arc
failure; they aim at providing routes faster than other modes for a
high proportion of the trips under a budget
constraint. \textcolor{black}{\citet{Archilla2013_HeuristicMaxCovering}
  study a similar problem and propose a heuristic approach as its
  solution methodology. \citet{Bucarey2020} study this problem setting
  to enhance its formulation and further introduce a partial covering
  problem by enforcing a lower bound on the ridership amount while
  minimizing the network design cost.  In these problems, user choices
  can be associated with the costs or the durations of the trips to
  represent their mode switching behavior.}  Due to the complexity in
modeling and solving these problems with respect to the dual
perspectives of transit agency and riders, these studies focus on
single-level formulations.

\textcolor{black}{To represent the travel behavior of the riders in
  transit systems, \citet{Ye2007} present the important factors in
  adoption decisions such as trip duration and the number of transfers
  of the proposed routes, along with the income levels of the
  riders. Additionally, \citet{Correa2011} discuss the importance of
  cost in mode selection if the riders are subject to the price of the
  suggested route.  To capture the mode selection behavior of the
  riders in a given origin-destination pair, all-or-nothing policies
  can be adopted for the mode decisions of all riders in that trip or
  logit models can be used to separate these riders
  \citep{Laporte2005_Logit}.  \citet{Chowdhury2016} provide a
  comprehensive review on the rider perspectives in public transit.
  Recently, \citet{Yan2021} study the travel behavior of the
  low-income riders in on-demand public transit systems as opposed to
  fixed public transit systems, and observe higher adoption
  preferences due to the higher flexibility and \textcolor{black}{access} provided
  by on-demand services. These studies highlight the importance of
  trip duration in determining the adoption behavior of the riders,
  which can be further impacted by the characteristics of the rider
  and route of the corresponding trip. These factors along with the
  transfer times and the costs of the trips can be integrated into the
  trip duration to obtain a combined metric in determining
  personalized travel choice functions \citep{Basciftci2020}.}

As should be clear at this point, the design of public transit systems
involve decision-making processes from multiple entities, including
transit agencies and riders \citep{Laporte2011SocioEconomic}. Bilevel
optimization is thus a key methodology to formulate these multi-player
optimization problems and it has been applied to several urban transit
network design problems \citep{Leblanc1986, Farahani2013_Review}. This
setting involves a leader who determines a set of decisions, and the
followers determine their actions under these decisions.
\citet{Fontaine2014} study the discrete network design problem where
the leader designs the network to reduce congestion under a budget
constraint and the riders search for the shortest path from their
origin to destination. \citet{Jia2012} and \citet{Yu2015} consider
this setting over multimodal transit networks with buses and cars;
they determine which bus legs are open and with which frequencies, and
ensure traffic equilibrium.  Bilevel optimization is also studied in
toll optimization problems over multicommodity transportation networks
by maximizing the revenues obtained through tolls in the leader
problem and obtaining the paths with minimum costs in the follower
problem \citep{Labbe1998, Brotcorne2001}. These studies are then
extended to a more general problem setting when the underlying network
is jointly optimized while considering the pricing aspect
\citep{Brotcorne2008}. \citet{PINTO20202} also apply bilevel
optimization to the joint design of multimodal transit networks
and shared autonomous mobility fleets. Here, the upper-level problem is
a transit network frequency setting problem that allows for the
removal of bus routes.

\citet{Colson2007} provide an overview of bilevel optimization
approaches with solution methodologies and discuss traffic equilibrium
constraints that may complicate the \textcolor{black}{network design
  problems} further when congestion is considered. \citet{Colson2005,
  Sinha2018} further present possible solution methodologies to
address bilevel optimization problems. Due to the complex nature of
the bilevel problems involving transportation networks, various
studies (e.g., \citet{Bianco2009, Jia2012, Kalashnikov2016}) focus on
developing heuristics as its solution methodology. On the other hand,
\citet{Gao2005, Fontaine2014, Yu2015} provide reformulation and
decomposition-based solution methodologies to provide exact solutions
for this class of problems. Despite this extensive literature on
bilevel optimization in transportation problems, personalized rider
preferences regarding transit routes have not been incorporated into
the network design. \textcolor{black}{As rider choices are neglected
  within the planning process, the latent demand, i.e., potential
  riders who can adopt the transit system, is disregarded, potentially
  leading to suboptimal network designs with lower adoptions.}  To our
knowledge, \citet{Basciftci2020} provide the first study that focuses
on this bilevel optimization problem by associating rider choices with
the cost and time of those trips in the ODMTS
system. \textcolor{black}{The leader problem optimizes the network
  design of the ODMTS, and the follower problems identify the optimum
  route of each trip based on their weighted cost and
  convenience. Additionally, riders have a personalized choice model
  to determine their travel mode by observing the suggested route.}
The studied problem considers the specific case where the transit
agency and riders subsidize the cost of the trips equally, leading
rider choices to be based on a combination of these cost and
convenience. However, if pricing is not equally subsidized between
these entities or rider preferences solely depend on the time of the
trips, then the problem becomes much more challenging to solve. To
address these challenges, this paper extends this line of research and
models rider preferences that depends on trip convenience for a
transit system with fixed ticket prices. Since this setting
substantially complicates exact solution methods, this paper studies
an exact decomposition method that exploits Benders optimality cuts,
combinatorial cuts, and dedicated valid inequalities strengthening the
combinatorial
cuts. \textcolor{black}{Section~\ref{sec:DiscussiontoCPAIORPaper}
  provides an extensive comparison and discussion of the two proposed
  models and highlights the contributions of this paper in comparison
  to existing studies.} This paper also contains an extensive
computational study that includes rider adoption, cost, revenue, and
\textcolor{black}{access to transit systems} aspects on various instances.

\section{The Bilevel Optimization Approach}
\label{sec:BilevelOptFormulation} 

This section presents a bilevel optimization approach for the ODMTS
design based on a game theoretic framework between the transit
agencies and riders. The transit agency is the leader who determines
the transit network design of the system, whereas the riders are the
followers who decide whether to adopt the transit system as their
travel mode. The proposed framework aims at designing the ODMTS
network while taking into account both existing transit riders and the
latent demand, i.e., riders who observe the system design and
performance, and decide their travel mode
accordingly. Section~\ref{sec:ProblemSetting} describes the problem
setting, Section ~\ref{sec:OptModel} presents the optimization model, {\color{black}Section ~\ref{sec:DiscussiontoCPAIORPaper} provides a discussion on the proposed framework,} 
and Section~\ref{sec:PreprocessingSteps} presents preprocessing steps
for dimensionality reduction. This proposed problem stays as close as
possible to the original setting of the ODMTS design \citep{Arthur2019}.

\subsection{Problem Setting} 
\label{sec:ProblemSetting}

The input {\color{black}for the ODMTS design} is defined in terms of a set $N$ of nodes associated with
bus stops, a subset $H \subseteq N$ of which are designated as
hubs. Each trip $r \in T$ has an origin stop $or^r \in N$, a
destination stop $de^r \in N$, and a number of riders taking that trip
$p^r \in \mathbb{Z}_{+}$. The time and distance between each pair $i,
j \in N$ are denoted by $t_{ij}$ and $d_{ij}$ respectively. These
parameters can be asymmetric but are assumed to satisfy the triangular
inequality. 
\textcolor{black}{Costs and inconvenience (e.g.,
travel time) are the two main aspects that transit agencies consider during network design. As the agencies generally operate under limited budget, it becomes critical to minimize cost. On the other hand, designing transit systems with better convenience not only improves the service for existing riders but also 
provides a more appealing mode choice for potential riders who may now decide to adopt the system when the duration of their suggested routes improves. Furthermore, adoption of additional riders increases the revenue for the transit agency. To this end,} the optimization problem uses a parameter
$\theta \in [0,1]$ to balance both objectives using a convex
combination. In particular, inconvenience is associated with the travel time
and multiplied by $\theta$, while travel cost is associated with the
travel distance and multiplied by $1 - \theta$.

Riders pay a fixed cost $\phi$ to use the transit system, irrespective
of their routes. \textcolor{black}{This fixed cost per rider becomes a
  revenue to the transit agency, which is captured as \[
  \varphi = (1 -
  \theta) \phi,
  \] in the leader objective for additional riders. If a
  leg between the hubs $h, l \in H$ is open, then the transit agency
  incurs an investment cost $\rho \, n \, d_{hl}$, where $\rho$ is the cost
  of using a bus per mile and $n$ is the number of buses operating in
  each open leg within the planning horizon. This cost is captured as \[
  \beta_{hl} = (1 - \theta) \rho \, n \, d_{hl}
  \]
  in the objective.}
Moreover, the transit agency incurs a service cost for each trip $r
\in T$ that consists of the weighted cost and inconvenience of using
bus legs between hubs and on-demand shuttle legs between bus
stops. More specifically, the weighted cost and inconvenience for an
on-demand shuttle between $i$ and $j$ 
is given
by \[
\gamma_{ij}^r = (1 - \theta) g \, d_{ij} + \theta t_{ij},
\] where
$g$ is the cost of using a shuttle per mile. Since the operating cost
of buses are already considered within the investment costs, each bus
leg between the hubs $h, l \in H$ in trip $r \in T$ only incurs an
inconvenience cost \[
\tau^r_{hl} = \theta (t_{hl} + s), 
\] where $s$ is
the average waiting time of a bus.

To represent the latent demand for the transit system, the set of
trips is partitioned into two groups: riders from \textcolor{black}{the
  trip set} $T' \subseteq T$ currently travel with their personal
vehicles, and riders from \textcolor{black}{the trip set} $T \setminus
T'$ currently use the transit system.  \textcolor{black}{The modeling
  assumes that existing transit riders will remain loyal to the ODMTS,
  given that case studies have demonstrated that ODMTS improves rider
  convenience for the vast majority of the trips and these riders
  might not have an alternative mode of transportation.}  Riders from
$T'$ may switch their travel mode from their personal vehicles to the
ODMTS, depending on the inconvenience of the route assigned to them.
Consequently, each trip $r \in T'$ is associated with a binary choice
model ${\cal C}^r$ that determines, given a proposed route, whether
its \textcolor{black}{riders adopt} the ODMTS. More precisely, given
route vectors $\bold{x^r}, \bold{y^r}$ for trip $r$, 
\textcolor{black}{which are described in more detail in Section \ref{sec:OptModel} and represent the utilized hub legs and on-demand shuttles respectively, }${\cal
  C}^r(\bold{x^r}, \bold{y^r})$ holds if trip $r$ adopts the
ODMTS. Since the price of the ODMTS is fixed, this paper assumes that
the choice model only depends on the trip inconvenience which is
captured by the function
\begin{equation} \label{eq:convenienceVariable}
f^r(\bold{x^r}, \bold{y^r}) = \sum_{h,l \in H}  (t_{hl} + s) x_{hl}^r + \sum_{i,j \in N}  t_{ij} y_{ij}^r.
\end{equation}
\textcolor{black}{In this choice model, waiting times are considered at
  every transfer point at hub locations to {\color{black}account for} the impact of
  transfers within the suggested route.  On the other hand, waiting
  time for on-demand shuttles is considered negligible as ride-sharing
  operations can be optimized in real-time using efficient algorithms
  (e.g., \citet{RileyRideSharing2019}) to obtain low waiting times.}
Moreover, the paper assumes that a rider will adopt the ODMTS if her
trip inconvenience in the transit system is not more than $\alpha^r$
times of her direct trip duration $t^r_{cur}$ (using her personal
vehicle), where $\alpha^r$ is a parameter associated with the
rider. More formally, the paper adopts the following choice model
\begin{equation} 
\label{eq:ChoiceFunctionTime}
{\cal C}^r(\bold{x^r}, \bold{y^r}) \equiv \mathbbm{1}(f^r(\bold{x^r}, \bold{y^r}) \leq \alpha^r \ t^r_{cur}).
\end{equation}

{\color{black}Before introducing the optimization model, it is useful
  to recall how ODMTS is designed and operated: (1) the transit agency
  designs the ODMTS to optimize a weighted combination of system cost
  and rider convenience; (2) when a rider requests an ODMTS trip during
  operation, she is presented by the ODMTS runtime system with the
  route that again optimizes a weighted combination of system cost
  and rider convenience; and (3) the rider then decides whether to
  adopt the proposed route based on her choice model or to drive with
  her own vehicle. The choice model of a rider is purely based on
  convenience, since the price of the ODMTS ride is fixed.
Section~\ref{sec:DiscussiontoCPAIORPaper} discusses this framework
further and, in particular, highlights the need for a bilevel
optimization.  Indeed, while a single-level optimization can be
formulated, it would enable the transit agency to propose arbitrarily
bad rides to users in order to avoid serving them.}

{\color{black}
\begin{table}[h]
\centering
\begin{tabular}{|rl}
\hline
\multicolumn{2}{|l|}{${\color{black}\textbf{Sets}}$}   \\ \hline
{\color{black}$N$}                        & \multicolumn{1}{l|}{{\color{black}Set of bus stops.}}                                              \\
{\color{black}$H$}                      & \multicolumn{1}{l|}{{\color{black}Set of potential hubs.}}                                               \\
{\color{black}$T$}                      & \multicolumn{1}{l|}{{\color{black}Set of all trips (existing trips and latent demand).}}                                               \\
{\color{black}$T'$}                     & \multicolumn{1}{l|}{{\color{black}Set of trips with choice (latent demand).}}                                               \\
\hline
\multicolumn{2}{|l|}{${\color{black}\textbf{Parameters}}$}   \\ \hline
{\color{black}$\theta$}                        & \multicolumn{1}{l|}{{\color{black}Weight factor for cost and inconvenience.}}                                              \\
{\color{black}$\beta_{hl}$}                      & \multicolumn{1}{l|}{\color{black}Weighted setup cost of opening the leg between hubs $h,l$.}  
 \\
{\color{black}$\tau_{hl}^r$}                      & \multicolumn{1}{l|}{\color{black}Weighted cost and inconvenience of the leg between hubs $h,l$ for trip $r$.} \\
{\color{black}$\gamma_{ij}^r$}                      & \multicolumn{1}{l|}{\color{black}Weighted cost and inconvenience of the on-demand shuttle between stops $i,j$ for trip $r$.} \\
{\color{black}$\varphi$}                      & \multicolumn{1}{l|}{\color{black}Weighted ticket price.} \\
{\color{black}$\phi$}                      & \multicolumn{1}{l|}{\color{black}Ticket price.} \\
{\color{black}$t_{ij}$}                      & \multicolumn{1}{l|}{\color{black}Travel time between stops $i,j$.} \\
{\color{black}$d_{ij}$}                      & \multicolumn{1}{l|}{\color{black}Travel distance between stops $i,j$.} \\
{\color{black}$s$}                      & \multicolumn{1}{l|}{\color{black}Average waiting time at hubs.} \\
\hline
\multicolumn{2}{|l|}{${\color{black}\textbf{Decision Variables}}$}     \\ \hline
{\color{black}$z_{hl}$}           & \multicolumn{1}{l|}{\color{black}$1$ if the leg from hubs $h$ to $l$ is open, and $0$ otherwise.}                        \\
{\color{black}$x_{hl}^r$}          & \multicolumn{1}{l|}{\color{black}$1$ if route of trip $r$ utilizes the leg from hubs $h$ to $l$, and $0$ otherwise.}                     \\ 
{\color{black}$y_{ij}^r$}          & \multicolumn{1}{l|}{\color{black}$1$ if route of trip $r$ utilizes an  on-demand shuttle from stops $i$ to $j$, and $0$ otherwise.}                     \\ 
{\color{black}$\delta^r$}          & \multicolumn{1}{l|}{\color{black}$1$ if riders of trip $r$ adopts the ODTMS, and $0$ otherwise.}                     \\
{\color{black}$b^r$}          & \multicolumn{1}{l|}{\color{black}Weighted cost and inconvenience of trip $r$.}                     \\
{\color{black}$f^r$}          & \multicolumn{1}{l|}{\color{black}Inconvenience of trip $r$.}                     \\
\hline
\end{tabular}
\caption{Problem parameters and decision variables.}
\label{tab:problemParameters}
\end{table}
}

\subsection{The Bilevel Optimization Model} 
\label{sec:OptModel}

The decision variables of the optimization model are as follows:
Binary variable $z_{hl}$ is 1 if the bus leg between the hubs $h,l \in
H$ is open. Additionally, for each trip $r \in T$, binary variables
$x_{hl}^r$ and $y_{ij}^r$ represent whether the route selected for
trip $r$ utilizes the bus leg between the hubs $h,l \in H$, and the
shuttle leg between the stops $i, j \in N$, respectively.
\textcolor{black}{Given a network design, variable $b^r$ corresponds to
  the weighted cost and inconvenience (i.e., trip duration) of trip $r
  \in T$ by considering the hub leg and on-demand shuttle components
  used in serving that trip. Similarly, variable $f^r$ is introduced
  in \eqref{eq:convenienceVariable} and represents solely the
  inconvenience of trip $r \in T$.}  The optimization model also uses
a binary decision variable $\delta^r$ for each trip $r \in T'$ to
represent whether its rider switches her travel mode to the ODMTS.
\textcolor{black}{Note that all riders of a trip $r \in T'$ are assumed
  to have the same adoption behavior with the same $\alpha^r$ value.}
\textcolor{black}{Table~\ref{tab:problemParameters} provides a summary of the \textcolor{black}{main} sets, parameters and decision variables used in the optimization model.}

\begin{figure}[!t]
\begin{subequations} \label{eq:upperLevelProblemUpdated2}
\begin{alignat}{1}
\min_{z_{hl}, b^r, \delta^r} \quad & \sum_{h,l \in H} \beta_{hl} z_{hl} + \sum_{r \in T \setminus T'} p^r b^r + \sum_{r \in T'} p^r \delta^r (b^r - \varphi) \label{eq:upperLevelObj} \\
\text{s.t.} \quad & \sum_{l \in H} z_{hl} = \sum_{l \in H} z_{lh} \quad \forall h \in H \label{eq:upperLevelConstr1} \\
& \delta^r = {\cal C}^r(\bold{x^r}, \bold{y^r}) \quad \forall r \in T' \label{eq:userChoiceModel} \\
& z_{hl} \in \{0,1\} \quad \forall h,l \in H  \label{eq:binaryConstraint} \\
& \delta^r \in \{0,1\} \quad \forall r \in T' \label{eq:continuousConstraint} 
\end{alignat}
\end{subequations}
where $(\bold{x^r}, \bold{y^r}, b^r)$ are a solution to the optimization problem
\begin{subequations}
\label{eq:lowerLevelProblem}
\begin{alignat}{1}
\lexmin_{x_{hl}^r, y_{ij}^r, b^r, f^r} \quad &  \langle b^r, f^r \rangle \label{eq:lowerLevelObj} \\
\text{s.t.} \quad
   & b^r = \sum_{h,l \in H} \tau_{hl}^r x_{hl}^r + \sum_{i,j \in N} \gamma_{ij}^r y_{ij}^r \label{eq:drDefinition} \\
   & f^r = \sum_{h,l \in H}  (t_{hl} + s) x_{hl}^r + \sum_{i,j \in N}  t_{ij} y_{ij}^r \label{eq:frDefinition} \\
   & \sum_{\substack{h \in H \\ \text{if } i \in H}} (x_{ih}^r - x_{hi}^r) + \sum_{j \in N}  (y_{ij}^r - y_{ji}^r) = \begin{cases}
     1 &, \text{if  } i = or^r \\
    -1 &, \text{if  } i = de^r \\
    0 &, \text{otherwise}
    \end{cases} \quad \forall i \in N \label{eq:minFlowConstraint} \\
  & x_{hl}^r \leq z_{hl} \quad \forall h,l \in H \label{eq:openFacilityOnlyAvailable} \\
  & x_{hl}^r \in \{0,1\} \quad \forall h,l \in H, y_{ij}^r \in \{0,1\} \quad \forall i,j \in N. \label{eq:integralityFlowConstr}
\end{alignat}
\end{subequations}
\caption{The Bilevel Optimization Model for ODMTS Design with Travel Mode Adoption.}
\label{fig:bilevel}
\end{figure}

The optimization model is given in Figure \ref{fig:bilevel}: it
consists of a leader model and a follower problem for each trip
$r$. The leader problem (Equations \eqref{eq:upperLevelObj}--
\eqref{eq:continuousConstraint}) determines the network design between
the hubs for the ODMTS whereas, \textcolor{black}{given this design,}
the follower problem (Equations
\eqref{eq:lowerLevelObj}--\eqref{eq:integralityFlowConstr}) identifies
routes for each trip $r \in T$ by utilizing the legs in this network
along with the on-demand shuttles \textcolor{black}{that can serve the
  first and last miles of the trip or provide a direct ride from its
  origin to destination}.

\textcolor{black}{The leader objective \eqref{eq:upperLevelObj} minimizes the sum of (i)
the investment cost of opening bus legs, (ii) the weighted cost and
inconvenience of the trips of the existing riders, and (iii) the weighted cost and
inconvenience minus 
revenues of those riders adopting the ODMTS. 
As existing transit riders are assumed to adopt the ODMTS, their constant revenue component is omitted in the objective.}
Constraint
\eqref{eq:upperLevelConstr1} guarantees weak connectivity between the
hubs by ensuring the sum of incoming and outgoing open legs to be
equal to each other for each hub.  \textcolor{black}{Although this
  formulation does not eliminate the potential of disconnected
  components in the network, the case studies under various demand
  patterns and parameter settings in Section \ref{sec:ODMTSDesign}
  always result in connected designs.}  Constraint
\eqref{eq:userChoiceModel} captures the mode choice of the riders in
$T'$ based on the ODMTS routes.

For a given trip $r$, the follower problem
\eqref{eq:lowerLevelProblem} minimizes the lexicographic objective
function $\langle b^r, f^r \rangle$, where $b^r$ represents the cost
and inconvenience of trip $r$ and $f^r$ breaks potential ties by
returning a most convenient route for the rider of trip $r$. Observe
that this \textcolor{black}{latter objective} is aligned with the
travel choice model. Constraint \eqref{eq:minFlowConstraint} enforces
flow conservation for the bus and shuttle legs used in trip
$r$. Constraint \eqref{eq:openFacilityOnlyAvailable} ensures that the
route only considers open bus legs.
Note that sub-objective $b^r$ contains sub-objective $f^r$ multiplied by
$\theta$, and the lexicographic objective breaks ties by choosing the
optimal value of $b^r$ with the smallest value of $f^r$.

\begin{prop} \label{lexicographicMinimumProp}
For any $\bold{z} \in \{0,1\}^{|H| \times |H|}$, a lexicographic
minimizer of problem \eqref{eq:lowerLevelProblem} exists and the
lexicographic minimum is unique.
\end{prop}

\noindent
This proposition follows because the feasible space of a follower
subproblem is not empty, since there is always a direct shuttle route from $or^r$ to $de^r$. Moreover, each component of the
objective is bounded from below.

Observe that, once a design $\bold{z}$ is chosen, the mode choice of
every rider is uniquely determined, which is important for
computational reasons. Moreover, the follower problem has a totally
unimodular constraint matrix, and can be solved as a linear program
using an objective of the form $M \ b^r + f^r$ for a suitably large
$M$. In the rest of the paper, a solution $\bold{z} \in \{0,1\}^{|H|
  \times |H|}$ is called an ODMTS design. Moreover, given two ODMTS
designs $\bold{z}^1$ and $\bold{z}^2$, $\bold{z}^1 \leq \bold{z}^2$
iff $z^1_{hl} \leq z^2_{hl}$ for all $h, l \in
H$. \textcolor{black}{This means that every bus leg that is open in
  $\bold{z}^1$ is also open in $\bold{z}^2$ with potentially more bus
  legs open in the latter design.} 

{\color{black}
\subsection{Discussion on the Proposed Model} 
\label{sec:DiscussiontoCPAIORPaper}

The model in Figure \ref{fig:bilevel} considers an optimization of an
on-demand multimodal network over a choice function for riders that
considers only convenience. This captures the reality of transit
systems as most of these systems are currently organized with fixed
pricing strategies and, as a result, preferences of the potential
riders can be based on the convenience of the suggested routes. Under
this setting, from the transit agency's perspective, cost and
convenience may be antagonistic to each other. Specifically, if only
convenience matters, then shuttles would be used for serving the
trips, increasing the cost of the ODMTS. On the other hand, the
network designer may decrease the cost for the network design by
opening new bus lines and benefit from economies of scale. These bus
lines may improve the convenience of some riders already using the bus
network.  But it may also worsen the convenience of some other
riders, who may not have direct shuttle trips anymore or may now have
shorter first/last shuttle legs. Those riders may thus decide not to
adopt the transit system because of the worse convenience. In effect,
the realistic setting adopted in the paper creates a non-monotonic
behavior in the design process, as opening or closing bus legs may
increase or decrease convenience of the riders. In turn, this behavior
further necessitates the bilevel structure of the optimization
model. Indeed, a single-level model would let the optimization choose
which route to propose to each rider and could therefore choose routes
that are so long that the rider will not adopt the system. The
optimization would then select which riders and neighborhoods it would
serve, and rejects those who are ``profitable'', defeating the purpose
of public transit and the need to serve underrepresented,
low-income communities. 
{\color{black}Section~\ref{sec:ComparisontoSingleLevel} in the Appendix provides 
the formulation for the single-level problem and illustrates this unfair behavior 
of the transit agency over a sample instance. These results demonstrate that
the model suggests longer routes for a subset of potential riders, so that they
do not adopt the transit system, because they are not profitable. This results
in significantly lower adoption ratios. This unfair behavior goes against the 
mission of transit agencies that generally aim at providing an equitable and unbiased
access to their system. This is precisely what the bi-level model achieves.}

This paper thus proposes a fundamentally different setting compared to
\citet{Basciftci2020}. Indeed, as discussed in the literature review,
\cite{Basciftci2020} study an optimization model where the objective
of the transit agency and the choice models of the riders are aligned
and consist of a convex combination of cost and
convenience. Specifically, in their study, mode choices depend on the
variable $b^r$, as opposed to the convenience $f^r$, with the choice
function $\mathbbm{1}(b^r(\bold{x^r}, \bold{y^r}) \leq \alpha^r
\ b^r_{cur})$, where $b^r_{cur}$ represents the weighted cost and
convenience of the rider's current trip using her personal vehicle.
Furthermore, the costs of on-demand shuttles are equally subsidized
between the transit agency and riders: the weighted cost and
convenience for an on-demand shuttle between $i$ and $j$ for both the
transit agency and riders is given by $(1 - \theta) \ \frac{g}{2}
\ d_{ij} + \theta t_{ij}$, which half the cost of the on-demand
shuttle component $g$. On the other hand, in this paper, the objective
for the transit agency, i.e., $\gamma^r_{ij}$ for trip $r$, is given
by $(1 - \theta) \ g \ d_{ij} + \theta t_{ij}$ and the riders pay a
fixed price for any trip. As a result, the choice models focus
exclusively on convenience but may differ obviously for different
classes of riders. This paper also models the additional revenues
coming from transit adoption in its objective function. 

The model has fundamental mathematical and computational consequences.
The alignment of the choice functions and the objective function in
\citet{Basciftci2020} ensures a desirable monotonicity property: as
more bus lines are open, the $b^r$ values improve.  This monotonic
relationship between the network design $\bold{z}$ and the $b^r$
values simplifies the combinatorial cuts that are added as a part of
the solution procedure to ensure the consistency between rider choices
and network design decisions as rider choices remain consistent with
changes in the designs. On the other hand, in this paper, adding bus
lines may improve or decrease convenience $f^r$, creating a
non-monotonic behavior that complicates the cut generation procedure.
As discussed later in the paper, the combinatorial cuts now need to be
lifted without these desirable monotonicity property. Section
\ref{sec:FinalComparisonCPAIOR} in the Appendix further discusses the
comparison of the two studies by highlighting the novel technical
results, the differences in the cut generation procedures, and the
case studies.}

\subsection{Preprocessing Steps}
\label{sec:PreprocessingSteps}

This section presents a number of preprocessing steps to simplify the
bilevel optimization problem.

\subsubsection{Linearization of the Leader Problem}
\label{sec:UBonFollowerProblemObj}

The objective function of the leader problem \eqref{eq:upperLevelObj}
includes bilinear terms $\delta^r b^r$ for all trips $r \in T'$. These
terms can be linearized with an exact McCormick reformulation since
$\delta^r$ is a binary variable. In particular, a bilinear term
$\delta^r b^r$ ($r \in T'$) in the objective function is replaced with
a new variable $\nu^r$, and the following constraints are added to the
leader problem:
\begin{subequations} 
\label{eq:objLinearization}
\begin{alignat}{1}
\nu^r & \leq M^r \delta^r \\
\nu^r & \leq b^r \\
\nu^r & \geq b^r - M^r (1 - \delta^r) \\
\nu^r & \geq 0,
\end{alignat}
\end{subequations}
where the term $M^r$ is an upper bound on the value of $b^r$. The 
following result is helpful in finding such a bound. 

\begin{prop} \label{monotoneTheorem}
  Let $r \in T$ and $({b^r_1}^*, {f^r_1}^*)$ and $({b^r_2}^*,
  {f^r_2}^*)$ be the optimal objective values of the follower problem
  under the ODMTS designs $\bold{z}^1$ and $\bold{z}^2$. If
  $\bold{z}^1 \leq \bold{z}^2$, then ${b^r_1}^* \geq {b^r_2}^*$.
\end{prop}
\begin{proof}{Proof:}
If $\bold{z}^1 \leq \bold{z}^2$, then $\bold{z}^2$ has at least as
many bus legs as $\bold{z}^1$. Hence, the feasible region of the
follower problem under $\bold{z}^1$ is a subset of the feasible region
under $\bold{z}^2$. $\square$
\end{proof}

\noindent
For a given ODMTS design and a trip $r$, the follower problem
\eqref{eq:lowerLevelProblem} returns a path of least cost and
inconvenience between $or^r$ and $de^r$. As a result, by Proposition
\ref{monotoneTheorem}, the ODMTS design with no bus leg gives an upper
bound on the value of $b^r$. Similarly, the ODMTS design with all legs
open returns a lower bound that can be inserted in the leader problem
to strengthen the formulation.

\subsubsection{Elimination of Arcs}
\label{sec:ArcElimination}

The follower problem \eqref{eq:lowerLevelProblem} considers all arcs
between nodes $i,j \in N$ for shuttle legs. However, only a subset of
these arcs are needed due to the triangular inequality on arc weights.
In particular, the follower problem needs only to consider arcs i)
from origin to hubs, ii) from hubs to destination, and iii) from
origin to destination. This set of arcs is denoted as $A^r$ in the
following. As a result, the bilevel optimization problem only uses the following
decision variables for each trip $r$:
\begin{align*}
y^r_{or^r h}, y^r_{h de^r} & \in \{0,1\} \quad \forall h \in H \\
y^r_{or^r de^r} & \in \{0,1\}.
\end{align*}

\section{Analytical Results on Trip Durations}
\label{SectionAnalyticalResultsonTripDurations}

This section presents analytical results that show how ODMTS designs
impact the duration of the routes proposed to riders. It focuses on
the general case where the trip origin and destination are not hub
locations: each such trip is of two possible forms: i) a combination
of legs including shuttle trips from origin to a hub and from a hub to
destination along with bus ride(s) between the hubs or ii) a direct
shuttle ride from origin to destination. Section
\ref{boundGenerationSection} derives upper and lower bounds on trip
durations when new arcs are added or existing arcs are removed from an
ODMTS design. Section \ref{liftingSection} identifies certain cases
where a trip duration does not worsen with the addition or removal of
arcs from a given design.  These results are used in Section
\ref{sec:solutionMethodologySection} in dedicated inequalities that
link ODMTS designs and rider choices. 

\subsection{Identification of bounds on the duration of the trips} 
\label{boundGenerationSection}

This section first derives upper bounds on trip durations when new
arcs are added to an ODMTS design. It then derives corresponding lower
bounds when arcs are removed from a design.

\begin{prop} \label{arcAdditionUBThm1}
Consider transit network design $\bold{z}^1$ and assume that the
optimal route for trip $r$ includes shuttle trips from origin $or^r$
to hub $m$ and from hub $n$ to destination $de^r$ with a trip time
$t^1$. For any network $\bold{z}^2 \geq \bold{z}^1$, the time $t^2$ of
the optimal route for trip $r$ admits the following upper bound: 
\begin{equation}
t^2 \leq t^1 + \frac{(1 - \theta)}{\theta} g \left(d_{or^rm}+ d_{nde^r} - \min_{h,l \in H} \left\{d_{or^rh} + d_{lde^r}\right\}\right) = UB^1.
\end{equation}
\end{prop}

\begin{proof}{Proof:}
Without loss of generality, assume that the optimal route of trip $r$
under design $\bold{z}^2$ includes the shuttle trips from origin
$or^r$ to hub $h'$ and from hub $l'$ to destination $de^r$.  Let
${b^r_1}^* = \theta t^1 + (1 - \theta) g (d_{or^rm} + d _{nde^r})$ and
${b^r_2}^* = \theta t^2 + (1 - \theta) g (d_{or^rh'} + d _{l'de^r})$
be the optimal objective function values under designs $\bold{z}^1$
and $\bold{z}^2$. If $\bold{z}^2 \geq \bold{z}^1$, then ${b^r_1}^*
\geq {b^r_2}^*$. It follows that:

\begin{align*}
\theta t^1 + (1 - \theta) g (d_{or^rm} + d _{nde^r}) & \geq \theta t^2 + (1 - \theta) g (d_{or^rh'} + d _{l'de^r}) \\
\theta t^1 + (1 - \theta) g \left(d_{or^rm} + d _{nde^r} - (d_{or^rh'} + d _{l'de^r})\right) & \geq \theta t^2 \\
t^1 + \frac{(1 - \theta)}{\theta} g \left(d_{or^rm} + d _{nde^r} - (d_{or^rh'} + d _{l'de^r})\right) & \geq t^2 \\
t^1 + \frac{(1 - \theta)}{\theta} g \left(d_{or^rm}+ d_{nde^r} - \min_{h,l \in H} \left\{d_{or^rh} + d_{lde^r}\right\}\right) & \geq t^2. \;\;\; \square
\end{align*}
\end{proof}

\begin{cor}
If $m$ is the closest hub to origin $or^r$ and $n$ is the closest hub
to destination $de^r$, then the upper bound in Proposition
\ref{arcAdditionUBThm1} reduces to $t^2 \leq t^1$.
\end{cor}

\noindent This corollary indicates that, if the route of a trip
includes shuttle components from its origin and destination to the
closest hubs, then addition of arcs only makes the duration of the
trip better. For example, if a rider is already adopting the ODMTS
under the initial design, then these riders will keep adopting the
system under the new design as the duration of the trip can only get
shorter.

\begin{prop} \label{arcAdditionUBThm2}
Consider ODMTS design $\bold{z}^1$ and assume that the optimal route
for trip $r$ is a direct shuttle trip with trip time $t^1$. For any
ODMTS design $\bold{z}^2 \geq \bold{z}^1$, the time $t^2$
of the optimal route for trip $r$ satisfies the following upper bound:
\begin{equation}
t^2 \leq \max \left\{t^1, t^1 + \frac{(1 - \theta)}{\theta} g \left(d_{or^r de^r} - \min_{h,l \in H} \left\{d_{or^rh} + d_{lde^r}\right\}\right)\right\} = UB^2.
\end{equation}
\end{prop}
\begin{proof}{Proof:}
Under $\bold{z}^2$, the optimal route for trip $r$ involves either a direct trip from origin $or^r$ to destination $de^r$ or a combination of rides involving shuttle trips from origin $or^r$ to some hub $h'$, from some hub $l'$ to destination $de^r$, and bus rides between hubs $h', l'$. In the first case, observe that $t^1$ is an upper bound on the trip duration $t^2$. In the second case, 
\begin{align*}
\theta t^1 + (1 - \theta) g d_{or^r de^r} & \geq \theta t^2 + (1 - \theta) g (d_{or^rh'} + d _{l'de^r}) \\
\theta t^1 + (1 - \theta) g \left(d_{or^r de^r} - (d_{or^rh'} + d _{l'de^r})\right) & \geq \theta t^2 \\
t^1 + \frac{(1 - \theta)}{\theta} g \left(d_{or^r de^r} - \min_{h,l \in H} \left\{d_{or^rh} + d_{lde^r}\right\}\right) & \geq t^2.
\end{align*}
Depending on $\bold{z}^2$, both cases are possible and the result follows.  $\square$
\end{proof}

\noindent
When $\bold{z}^1$ has no hub legs open, the optimal route for trip $r$ takes time $t_{or^r
  de^r}$. Therefore, for any network $\bold{z}^2 \geq \bold{z}^1$, the upper bound using Proposition~\ref{arcAdditionUBThm2} becomes
\begin{equation}
t^2 \leq \max \left\{t_{or^r de^r}, t_{or^r de^r} + \frac{(1 - \theta)}{\theta} g \left(d_{or^r de^r} - \min_{h,l \in H} \left\{d_{or^rh} + d_{lde^r}\right\}\right)\right\}.
\end{equation}
If this upper bound value is duration of the direct route, then the
trip must be served by an on-demand shuttle. The following corollary
can thus be used as a pre-processing step to identify direct shuttle
trips.

\begin{cor} \label{corrollaryDirectTripCondition}
For any trip $r \in T$, if $\min_{h,l \in H} \left\{d_{or^rh} +
d_{lde^r}\right\} \geq d_{or^r de^r}$, then the trip will be served
with on-demand shuttles only.
\end{cor}
\begin{proof}{Proof:}
The proof is by contradiction. Suppose that, trip $r$ is served with
on-demand shuttles to and from hubs, and bus leg(s) between hubs under
a network $\bold{z}^2$ where $\bold{z}^2 \geq \bold{z}^1$. Without
loss of generality, assume that the origin is connected to hub $m$ and
hub $n$ is connected to the destination. Then, $d_{or^r m} + d_{n
  de^r} \geq \min_{h,l \in H} \left\{d_{or^rh} + d_{lde^r}\right\}
\geq d_{or^r de^r}$. Moreover, the time of this route is at least the
time of the direct trip by the triangle inequality, contradicting the
hypothesis by definition of $b^r$. $\square$
\end{proof}

The next results derive lower bounds on trip durations. 

\begin{prop} \label{arcSubtractionLBTheorem1}
Consider ODMTS design $\bold{z}^1$, and assume that the optimal route
for trip $r$ includes the shuttle trips from origin $or^r$ to hub $m$
and from hub $n$ to destination $de^r$ with a trip time $t^1$. For any
design $\bold{z}^2$ such $\bold{z}^1 \geq \bold{z}^2$, the time $t^2$
of the optimal route for trip $r$ has a lower bound as
\begin{equation}
t^2 \geq t^1 + \frac{(1 - \theta)}{\theta} g \left(d_{or^rm} + d_{nde^r} - \max \left \{\max_{h,l \in H} \left\{d_{or^rh} + d_{lde^r} \right\}, d_{or^r de^r} \right\} \right) = LB^1.
\end{equation}
\end{prop}

\begin{proof}{Proof:}
Observe first that the optimum $b^r$ value for trip $r$ under
$\bold{z}^1$ is greater than or equal to the corresponding value under
network design $\bold{z}^2$. Without loss of generality, assume that
the optimum route of trip $r$ under design $\bold{z}^2$ includes either the
shuttle trips from origin $or^r$ to hub $h'$ and from hub $l'$ to
destination $de^r$, or a direct shuttle trip from origin $or^r$ to
destination $de^r$. In the first case,
\begin{align*}
\theta t^1 + (1 - \theta) g (d_{or^rm} + d _{nde^r}) & \leq \theta t^2 + (1 - \theta) g (d_{or^rh'} + d _{l'de^r}) \\
t^1 + \frac{(1 - \theta)}{\theta} g \left(d_{or^rm} + d _{nde^r} - (d_{or^rh'} + d _{l'de^r})\right) & \leq t^2 \\
t^1 + \frac{(1 - \theta)}{\theta} g \left(d_{or^rm}+ d_{nde^r} - \max_{h,l \in H} \left\{d_{or^rh} + d_{lde^r}\right\}\right) & \leq t^2.
\end{align*}
In the second case,
\begin{equation*}
t^1 + \frac{(1 - \theta)}{\theta} g \left(d_{or^rm}+ d_{nde^r} - d_{or^r de^r} \right) \leq t^2,
\end{equation*}
completing the proof. $\square$
\end{proof}

\begin{prop} \label{arcSubtractionLBTheorem2}
Consider ODMTS design $\bold{z}^1$, and assume that the optimal route
for trip $r$ is a direct shuttle trip from origin $or^r$ to
destination $de^r$ with a trip time~$t^1$. For any network
$\bold{z}^2$, $\bold{z}^1 \geq \bold{z}^2$, the time $t^2$ of the
optimum route for trip $r$ will be $t^2 = t^1 = LB^2$.
\end{prop}

\begin{proof}{Proof:}
As the feasible solutions under $\bold{z}^2$ is a subset of the
feasible solutions under $\bold{z}^1$, the optimum route of trip $r$ with
respect to the follower problem will remain as a direct shuttle trip
from origin $or^r$ to destination $de^r$. $\square$
\end{proof}

\subsection{Specific Network Designs}
\label{liftingSection}

This section presents two specific but important cases where the
duration of the studied trip cannot become worse when more bus legs
are added. The first case considers a trip route where shuttles
connect the origin and destination to hubs and where additional arcs
do not make closer hubs available. Given ODMTS design~$\bold{z}$,
define the set of \textit{active} hubs $\mathcal{H}(\bold{z}) = \{h
\in H: \sum_{l \in H} z_{hl} > 0\}$. Due to the weak connectivity
constraint~\eqref{eq:upperLevelConstr1}, $\sum_{l \in H} z_{hl} > 0$
implies $\sum_{l \in H} z_{lh} > 0$ for all $h \in H$. Define the
following minimum distances from/to node $i \in N$ to/from any active
hub under $\bold{z}$ as $\overrightarrow{d}_i^{\min} (\bold{z}) :=
\min_{h \in \mathcal{H}(z)}\{d_{ih}\}$ and $\overleftarrow{d}_i^{\min}
(\bold{z}) := \min_{h \in \mathcal{H}(z)}\{d_{hi}\}$. Finally, define
$\overrightarrow{W}_{i} (z) = \{h \in H \setminus \mathcal{H}(z):
d_{ih} < \overrightarrow{d}_i^{\min} (z)\}$ and $\overleftarrow{W}_{i}
(z) = \{h \in H \setminus \mathcal{H}(z): d_{hi} <
\overleftarrow{d}_i^{\min} (z)\}$ as the set of non-active hubs that
are closer to the origin and destination than the active hubs
respectively. The next proposition shows that, if the non-active hubs
closer to the origin and destination of a trip $r$ in the current
design remain inactive in a larger design, the duration of trip $r$
can only improve. 

\begin{prop} \label{arcAdditionLiftingThm}
Consider ODMTS design $\bold{z}^1$, and assume that the optimal
route for trip $r$ includes the shuttle trips from origin $or^r$ to hub
$m$, and from hub $n$ to destination $de^r$, with a trip time
$t^1$. If $m$ and $n$ are the closest active hubs to the origin and
destination, i.e., $d_{or^r m} = \overrightarrow{d}_{or^r}^{\min}
(\bold{z}^1)$ and $d_{n de^r} = \overleftarrow{d}_{de^r}^{\min}(\bold{z}^1)$, then
for any network design $\bold{z}^2$ satisfying
\begin{align*}
\bold{z}^2 \in \{\bold{z} \in \{0,1\}^{|H| \times |H|}: & z_{hl} = 1 \>\> \forall \> (h,l) \> s.t. \> z^1_{hl} = 1, \\
& \sum_{l \in H} z_{hl} = 0 \>\> \forall \> h \in \overrightarrow{W}_{or^r} (\bold{z}^1), \\
& \sum_{l \in H} z_{hl} = 0 \>\> \forall  \> h \in \overleftarrow{W}_{de^r} (\bold{z}^1) \},
\end{align*}
then the time $t^2$ of the optimal route for trip $r$ in $\bold{z}^2$ satisfies $t^2 \leq t^1$.
\end{prop}

\begin{proof}{Proof:}
By definition of $\bold{z}^2$, $\overrightarrow{d}_{or^r}^{\min}
(\bold{z}^1) = \overrightarrow{d}_{or^r}^{\min} (\bold{z}^2)$ and
$\overleftarrow{d}_{de^r}^{\min} (\bold{z}^1) =
\overleftarrow{d}_{de^r}^{\min} (z^2)$. This implies that $d_{or^r h}
\geq \overrightarrow{d}_{or^r}^{\min} (\bold{z}^2)$ and $d_{h de^r}
\geq \overleftarrow{d}_{de^r}^{\min} (\bold{z}^2)$ for all hubs $h \in
\mathcal{H}(\bold{z}^2)$. Since the cost only depends on the distance
of the shuttle rides, the cost of the optimal route under $\bold{z}^1$
is $g (\overrightarrow{d}_{or^r}^{\min} (\bold{z}^1) +
\overleftarrow{d}_{de^r}^{\min} (\bold{z}^1))$, and the corresponding cost
under  $\bold{z}^2$ become $g (d'_1 +
d'_2)$, where $d'_1 \geq \overrightarrow{d}_{or^r}^{\min} (z^1)$ and
$d'_2 \geq \overleftarrow{d}_{de^r}^{\min} (z^1)$. Since the latter
cost is greater than or equal to the former one, and $z^2 \geq z^1$,
it must be the case that $t^2 \leq t^1$. $\square$
\end{proof}

The next result identifies the set of arcs whose removal from the transit design do not impact the duration of the associated trip. 

\begin{prop} \label{arcRemovalLiftingThm}
Consider design $\bold{z}^1$, and assume that the optimal route of
trip $r$ takes time $t^1$. If design $\bold{z}^2$ is obtained from
$\bold{z}^1$ by removing some arcs that are not used on the optimal
route for $r$, then the trip duration for $r$ under $\bold{z}^2$
remains $t^1$.
\end{prop}

\section{Solution Methodology}
\label{sec:solutionMethodologySection}

This section proposes a solution methodology that {\color{black}decomposes} the
bilevel problem \eqref{eq:upperLevelProblemUpdated2} into a master
problem and subproblems. The approach combines a traditional Benders
decomposition \textcolor{black}{\citep{Benders1962}} to generate optimality cuts with combinatorial Benders
cuts to reconcile rider choices in the master problem with those
induced by the optimal routes in the follower subproblems. \textcolor{black}{In that sense, it is reminiscent of logical Benders and Branch-and-Check methods pioneered in   \citet{Hooker2002_IJOC,Thorsteinsson2001,HookerOttoson_2003, Hooker2007}.} More specifically, the master problem consists of the leader problem with
variables $(\{z_{hl}\}_{h,l \in H}, \{\delta^r\}_{r \in T'},
\{b^r\}_{r \in T})$ where the rider choice constraint
\eqref{eq:userChoiceModel} is relaxed. In each iteration, the follower
subproblems are solved to generate optimality cuts on variables $b^r$.
In addition, combinatorial cuts are introduced to guarantee the
consistency between ${\cal C}^r(\bold{x}^r,\bold{y}^r)$ and the master
variable $\delta^r$. These ``basic'' combinatorial cuts are further
improved using the results of Section
\ref{SectionAnalyticalResultsonTripDurations}. The proposed
decomposition algorithm converges when the lower bound obtained by the
master problem, and the upper bound constructed from the feasible
solutions of the subproblems are close enough.

The rest of this
section formally introduces the decomposition algorithm along with the
several enhancements. Section \ref{sec:RelaxedMaster} and Section
\ref{sec:Subproblem} present the master problem and the Benders
subproblems. Section \ref{sec:CutGeneration} proposes the cut
generation procedure for the optimality cut and the combinatorial cuts
for coupling the choice model and the network design.
Section
\ref{sec:DecAlgSummary} summarizes the decomposition algorithm, and
proves its finite convergence. Section
\ref{sec:ValidIneq} proposes valid inequalities that enforce the
relationship between the ODMTS designs and the rider choices. Finally, Section
\ref{sec:ParetoOptCuts} discusses Pareto-optimal cut generation
procedure for enhancing the performance of the solution methodology.

\subsection{Master Problem}
\label{sec:RelaxedMaster}

\textcolor{black}{To formally present the decomposition algorithm, the bilevel problem \eqref{eq:upperLevelProblemUpdated2} can be equivalently written as in the following form,}
{\color{black}
\begin{subequations} \label{eq:equivalentModel}
\begin{alignat}{1}
\min \quad & \sum_{h,l \in H} \beta_{hl} z_{hl} + \sum_{r \in T \setminus T'} p^r b^r + \sum_{r \in T'} p^r \delta^r (b^r - \varphi) \\
\text{s.t.} \quad & \eqref{eq:upperLevelConstr1}, \eqref{eq:binaryConstraint}, \eqref{eq:continuousConstraint}, \notag \\
& \mathcal{L'}^r (z, \delta^r) \geq 0 \quad \forall r \in {T'}, \label{eq:consistencyCutSet} \\
& \mathcal{L}^r (z, b^r) \geq 0 \quad \forall r \in T \label{eq:optimalityCutSet}.
\end{alignat}
\end{subequations}
}

\noindent
\textcolor{black}{The constraint set $\mathcal{L'}^r (z, \delta^r)$ in \eqref{eq:consistencyCutSet} corresponds to all combinatorial cuts that ensure the consistency between the network design and the choice variables, and the constraint set $\mathcal{L}^r (z, b^r)$ in \eqref{eq:optimalityCutSet} provide an explicit formulation of the follower problem, as traditionally done in deriving Benders decomposition methods. In particular, these cuts provide lower bounds on the $b^r$ values based on the follower problem. All of the cuts in \eqref{eq:consistencyCutSet} and \eqref{eq:optimalityCutSet} can be precomputed to obtain an equivalent formulation,
but they add exponentially many constraints. Thus, the proposed decomposition algorithm starts with a subset of them and dynamically adds the corresponding constraints as new network designs are identified, along with the addition of valid inequalities based on the analytical results on trip durations.}

\textcolor{black}{To this end, the initial master problem \eqref{eq:relaxedMasterProblemInitial} can be formulated as a relaxation of the problem~\eqref{eq:equivalentModel}:}
{\color{black}
\begin{subequations} \label{eq:relaxedMasterProblemInitial}
\begin{alignat}{1}
\min \quad & \sum_{h,l \in H} \beta_{hl} z_{hl} + \sum_{r \in T \setminus T'} p^r b^r + \sum_{r \in T'} p^r (\nu^r - \delta^r \varphi) \label{eq:masterObjectiveFunction} \\
\text{s.t.} \quad & \eqref{eq:upperLevelConstr1}, \eqref{eq:binaryConstraint}, \eqref{eq:continuousConstraint}, \eqref{eq:objLinearization}. \notag 
\end{alignat}
\end{subequations}
}
\noindent
At each iteration of the algorithm, the relaxed master problem
\eqref{eq:relaxedMasterProblemInitial} determines an ODMTS design to
be evaluated by the subproblems. Benders cuts and combinatorial
cuts are then added to this problem following the procedure proposed in
Section \ref{sec:CutGeneration} \textcolor{black}{along with the valid inequalities introduced in Section \ref{sec:ValidIneq}} to ensure optimality and
consistency between the rider choices in the master problem and the
follower routes.

\subsection{Subproblem for Each Trip}
\label{sec:Subproblem}


Given a transit network design solution $\{\bar{z}_{hl}\}_{h,l \in H}$
obtained by the master problem, the subproblem for each trip $r$ can
be formulated using the follower problem~\eqref{eq:lowerLevelProblem}
over the objective function $\hat{b}^r = M b^r + f^r$ and its
associated coefficients {\color{black} $\hat{\tau}_{hl}^r := M \tau_{hl}^r + t_{hl} + s$ and $\hat{\gamma}_{ij}^r := M \gamma_{ij}^r + t_{ij}$ $\hat{\tau}_{hl}$}.
The resulting problem can be formulated as follows: 
\begin{subequations}  \label{eq:subproblem}
\begin{alignat}{1}
\min \quad & \sum_{h,l \in H}  \hat{\tau}_{hl}^r x_{hl}^r + \sum_{i,j \in A^r} \hat{\gamma}_{ij}^r y_{ij}^r  \label{eq:lowerLevelObjSubproblem} \\
\text{s.t.} \quad & \sum_{\substack{h \in H \\ \text{if } i \in H}} (x_{ih}^r - x_{hi}^r) + \sum_{i,j \in A^r}  (y_{ij}^r - y_{ji}^r) = \begin{cases}
1 &, \text{if  } i = or^r \\
-1 &, \text{if  } i = de^r \\
0 &, \text{otherwise}
\end{cases}
\quad \forall i \in N \label{eq:minFlowConstraintSubprob} \\
& x_{hl}^r \leq \bar{z}_{hl} \quad \forall h,l \in H \label{eq:openFacilityOnlyAvailableSubprob} \\
& 0 \leq x_{hl}^r \leq 1, \quad \forall h,l \in H, 0 \leq y_{ij}^r \leq 1 \quad \forall i,j \in A^r. \label{eq:integralityFlowConstrSubprob}
\end{alignat}
\end{subequations}

\noindent
The model exploits the totally unimodular property of the follower
problem under a given binary solution $\{\bar{z}_{hl}\}_{h,l \in H}$
and uses the arc set $A^r$, eliminating the unnecessary arcs for the
on-demand shuttles. The dual of subproblem \eqref{eq:subproblem} is
expressed in terms of the dual variables $u^r_{i}$ and $v_{hl}^r$ that
correspond to constraints \eqref{eq:minFlowConstraintSubprob} and
\eqref{eq:openFacilityOnlyAvailableSubprob}:
\begin{subequations} \label{eq:dualLowerLevel}
\begin{alignat}{1}
\max \quad & (u_{or^r}^r - u_{de^r}^r) - \sum_{h,l \in H} \bar{z}_{hl} v^r_{hl} \\
\text{s.t.} \quad & u_h^r - u_l^r - v_{hl}^r \leq \hat{\tau}_{hl}^r \quad \forall h,l \in H \label{eq:dualLowerLevelConstr1} \\
& u_i^r - u_j^r \leq \hat{\gamma}_{ij}^r \quad \forall i,j \in A^r \label{eq:dualLowerLevelConstr2} \\
& u_i^r \geq 0 \quad \forall i \in N, v_{hl}^r \geq 0 \quad \forall h,l \in H. \label{eq:dualLowerLevelConstr3}
\end{alignat}
\end{subequations}
Note the primal subproblem \eqref{eq:subproblem} is always feasible
and bounded as each trip can be served by a direct shuttle
trip. Therefore, the dual subproblem \eqref{eq:dualLowerLevel} is
feasible and bounded as well. Benders optimality cuts in the form
\begin{equation} \label{eq:optimalityCutLowerLevel} 
d^{r} \geq (\bar{u}^{r}_{or^{r}} - \bar{u}^{r}_{de^{r}}) - \sum_{h,l \in H} z_{hl} \bar{v}_{hl}^{r}
\end{equation}
are thus added to the master problem at each iteration using the
optimal solution $(\bar{u}^r, \bar{v}^r)$ of the dual subproblem.

\subsection{Cut Generation Procedure}
\label{sec:CutGeneration}

This section presents how to achieve the consistency of the rider
choices in the master problem and those induced by the subproblems.

\begin{dfn}[Choice Consistency]
For a given trip $r$, the solution values $\{\bar{z}_{hl}\}_{h,l \in
  H}$ and $\bar{\delta}^{r}$ of the master problem are
\textit{consistent} with an optimal solution $(\bold{\bar{x}^r},
\bold{\bar{y}^r}, \bar{b}^r)$ of the follower problem
\eqref{eq:lowerLevelProblem} under the design $\{\bar{z}_{hl}\}_{h,l
  \in H}$ if $\bar{\delta}^{r} = {\cal C}^r(\bold{\bar{x}^r},
\bold{\bar{y}^r})$.
\end{dfn}

\noindent
To ensure choice consistency \textcolor{black}{between the choice variable $\delta^{r}$ and the evaluated choice function ${\cal C}^r$ under a given network design $\bold{z}$}, two possible cases must be considered:
\begin{enumerate}
\item Solution values $\{\bar{z}_{hl}\}_{h,l \in H}$ and $\bar{\delta}^{r}$ are \textit{inconsistent} with ${\cal C}^r(\bold{\bar{x}^r}, \bold{\bar{y}^r})$ when
\begin{enumerate}
\item $\bar{\delta}^{r} = 1$ and ${\cal C}^r(\bold{\bar{x}^r}, \bold{\bar{y}^r}) = 0$;
\item $\bar{\delta}^{r} = 0$ and ${\cal C}^r(\bold{\bar{x}^r}, \bold{\bar{y}^r}) = 1$.
\end{enumerate}
\item Solution values $\{\bar{z}_{hl}\}_{h,l \in H}$ and $\bar{\delta}^{r}$ are \textit{consistent} with ${\cal C}^r(\bold{\bar{x}^r}, \bold{\bar{y}^r})$.
\end{enumerate}

\noindent
By Proposition~\ref{lexicographicMinimumProp}, the lexicographic
minimum of problem \eqref{eq:lowerLevelProblem} is unique and hence
the routes of the lexicographic minimizers have the same cost and
inconvenience under a given ODMTS design.  Therefore, it is sufficient
to relate the rider choices with the ODMTS design to ensure the
consistency in these decisions. In particular, the first inconsistency
(case 1(a)) can be eliminated with the combinatorial cut \textcolor{black}{\eqref{eq:case2Cut} by ensuring $\delta^r$ to be 0 under the design $\bar{z}$.}
\begin{equation} \label{eq:case2Cut}
\sum_{(h,l):\bar{z}_{hl}=0} z_{hl} + \sum_{(h,l):\bar{z}_{hl}=1} (1-z_{hl}) \geq \delta^{r}.
\end{equation}


\noindent
The second inconsistency (case 1(b)) can be eliminated with the cut \textcolor{black}{\eqref{eq:case3Cut} by ensuring $\delta^r$ to be 1 under the design $\bar{z}$.}
\begin{equation} \label{eq:case3Cut}
\sum_{(h,l):\bar{z}_{hl}=0} z_{hl} + \sum_{(h,l):\bar{z}_{hl}=1} (1-z_{hl}) + \delta^{r} \geq 1.
\end{equation}


\noindent
Combinatorial cuts \eqref{eq:case2Cut} and \eqref{eq:case3Cut}
guarantee the consistency between the rider choice variables and the
choices induced by $\bar{\bold{z}}$. We can further strengthen these
cuts by exploiting the properties of the choice
model~\eqref{eq:ChoiceFunctionTime}.  Based on the analyses in Section
\ref{SectionAnalyticalResultsonTripDurations}, it is possible to add
new valid inequalities to the master problem at each iteration.

\begin{thm}
{\color{black}
Problem \eqref{eq:equivalentModel} is equivalent to the original Problem in Figure~\ref{fig:bilevel}.
}
\end{thm}

{\color{black}
\begin{proof}{Proof:}
Combinatorial cuts \eqref{eq:case2Cut} and \eqref{eq:case3Cut} constitute the consistency cut set \eqref{eq:consistencyCutSet}, whereas Constraint \eqref{eq:optimalityCutSet} represents the cuts \eqref{eq:optimalityCutLowerLevel}. Since $b^r$ is multiplied by a non-negative coefficient in the objective of the leader problem in Figure~\ref{fig:bilevel} and there are finitely many cuts in the form \eqref{eq:optimalityCutLowerLevel}, \eqref{eq:case2Cut}, \eqref{eq:case3Cut}, Problem \eqref{eq:equivalentModel} is equivalent to the original problem. 
\Halmos
\end{proof}
}

\subsection{The Decomposition Algorithm}
\label{sec:DecAlgSummary}


\textcolor{black}{With these definitions in place, it is possible to present the decomposition algorithm, which is summarized in Algorithm~\ref{alg:bilevelNetworkDecomposition}. The algorithm is guaranteed to converge to an optimal solution of Problem \eqref{eq:equivalentModel}.}

\begin{algorithm}
\caption{Decomposition Algorithm}
\label{alg:bilevelNetworkDecomposition}
\begin{algorithmic}[1]
\State Set $LB = -\infty$,  $UB = \infty$, $z^*=\emptyset$. 
\While {$UB > LB + \epsilon$}
\State Solve the relaxed master problem \eqref{eq:relaxedMasterProblemInitial} and obtain the solution ($\{\bar{z}_{hl}\}_{h,l \in H}$, $\{\bar{\delta}^r\}_{r \in T'}$, $\{\bar{b}^r\}_{r \in T})$.  
\State Update $LB$.
\ForAll{$r \in T$}
\State Solve the subproblem \eqref{eq:dualLowerLevel} under $\{\bar{z}_{hl}\}_{h,l \in H}$ and obtain $({b^r}^*, {f^r}^*)$. 
\State Add optimality cut in the form \eqref{eq:optimalityCutLowerLevel} to the relaxed master problem \eqref{eq:relaxedMasterProblemInitial}. 
\EndFor
\ForAll{$r \in T'$}
\If{$\{\bar{z}_{hl}\}_{h,l \in H}$ and $\bar{\delta}^{r}$ are  inconsistent with ${\cal C}^r(\bold{\bar{x}^r}, \bold{\bar{y}^r})$}
\State Add cuts in the form \eqref{eq:case2Cut} or \eqref{eq:case3Cut} to the relaxed master problem. \label{algorithmNogoodCutStep}
\EndIf
\State Add cuts discussed in Section \ref{sec:ValidIneq} if the {\color{black}sufficient} conditions are satisfied.
\If{${\cal C}^r(\bold{x^r}, \bold{y^r})$ is 1} 
\State Set $\hat{\delta}^r = 1$.
\Else
\State Set $\hat{\delta}^r = 0$.
\EndIf
\EndFor
\State $\widehat{UB} = \sum_{h,l \in H} \beta_{hl} \bar{z}_{hl} + \sum_{r \in T \setminus T'} p^r {b^r}^* + \sum_{r \in T'} p^r \hat{\delta}^r ({b^r}^* - \varphi)$.
\If{$\widehat{UB} < UB$} 
\State Update $UB$ as $\widehat{UB}$, $z^* = \bar{z}$. 
\EndIf
\EndWhile
\end{algorithmic}
\end{algorithm}

\begin{prop}
Algorithm \ref{alg:bilevelNetworkDecomposition} converges \textcolor{black}{to an optimal solution of Problem \eqref{eq:equivalentModel} in finitely many iterations}. 
\end{prop}
\begin{proof}{Proof:}
\textcolor{black}{First observe that there are finitely many combinatorial cuts
\eqref{eq:case2Cut} and \eqref{eq:case3Cut} that can be added to
ensure the relationship between network design and rider preferences
as all variables are binary. Similarly, there are finitely many 
optimality cuts of the form \eqref{eq:optimalityCutLowerLevel}, since
there are finitely many vertices in the dual follower subproblems. Hence Algorithm \ref{alg:bilevelNetworkDecomposition} is guaranteed to terminate.}

\textcolor{black}{It remains to show that it terminates with an optimal solution. Observe that the master problem provides a lower bound to Problem \eqref{eq:equivalentModel}, since it contains only a subset of the cuts. Moreover, at each iteration, Algorithm \ref{alg:bilevelNetworkDecomposition} computes a valid upper bound $\widehat{UB}$. If $\bar{b}^r = b^{r^*}$ for all $r \in T \setminus T'$, $\bar{\delta}^{r} = \hat{\delta}^r$ for all $r \in T'$, and $\bar{b}^r = b^{r^*}$ for all $r \in T'$ such that $\bar{\delta}^{r}=1$, the upper bound and the lower bound are the same, and the algorithm terminates with an optimal solution. Otherwise, it suffices to show that the algorithm generates at least one new cut. For $r \in T \setminus T'$, if $\bar{b}^r$ in the master problem is smaller than $b^{r^*}$, then the algorithm generates a new optimality cut (line 7). For $r \in T'$, if $\bar{\delta}^{r} \neq \hat{\delta}^r$, then the algorithm generates a new cut in line 9--10. If the choices are consistent and rider $r$ adopts the system (i.e., $\bar{\delta}^{r} = 1$), then the algorithm generates a new optimality cut if $\bar{b}^r$ in the master problem is smaller than $b^{r^*}$ (line 7 again). This concludes the proof. }
\Halmos
\end{proof}

\subsection{Valid Inequalities}
\label{sec:ValidIneq}

This section proposes valid inequalities for the studied problem
\eqref{eq:upperLevelProblemUpdated2} to strengthen the relationship
between transit network design and rider choice variables. The first
result utilizes the upper bound values on the duration of the trips.

\begin{lem} \label{arcAdditionUBProp}
For ODMTS design~$\bold{z}^1$, consider the upper bound $UB$ in
Propositions \ref{arcAdditionUBThm1} and \ref{arcAdditionUBThm2}. If a
rider of trip $r$ adopts the transit system under $\bold{z}^1$, and $UB \leq \alpha^r t^r_{cur}$, then the rider also adopts the
ODMTS under any design $\bold{z}^2$ such that $\bold{z}^1 \leq \bold{z}^2$, 
\end{lem}

\noindent
Lemma~\ref{arcAdditionUBProp} allows for the design of combinatorial
cuts \textcolor{black}{that strengthen the consistency cuts introduced
  in \eqref{eq:case3Cut}, by exploiting the property that a rider
  keeps adopting the system under any design with at least the bus
  legs open in $\bold{z}^1$.}

\begin{prop}
For a given transit network design~$\bold{z}^1$, if the condition in
Lemma \ref{arcAdditionUBProp} holds for trip $r$, then the
consistency cut becomes
\begin{equation} \label{eq:case2CutStronger}
\sum_{(h,l):z^1_{hl}=1} (1 - z_{hl}) + \delta^{r} \geq 1.
\end{equation}
\end{prop}

\noindent
The second result exploits the lower bound values on the duration of the trips.

\begin{lem} \label{arcSubtractionLBProp}
For design~$\bold{z}^1$, consider the lower bound $LB$ on trip
duration as derived in Propositions \ref{arcSubtractionLBTheorem1} and
\ref{arcSubtractionLBTheorem2}.  If a rider of trip $r$ \textcolor{black}{does} not adopt
the ODMTS under $\bold{z}^1$, and $LB \geq \alpha^r t^r_{cur}$, then
the rider will not adopt the ODMTS under any network design
$\bold{z}^2$ such that $\bold{z}^1 \geq \bold{z}^2$.
\end{lem}

\noindent
Lemma \ref{arcSubtractionLBProp} enables the derivation of
combinatorial cuts \textcolor{black}{that strengthen the consistency
  cuts introduced in \eqref{eq:case2Cut}, by benefiting from the
  conditions that the riders continue using their personal vehicles
  under any design with at most the bus legs open in $\bold{z}^1$.}

\begin{prop}
For a given design~$\bold{z}^1$, if the condition in Lemma \ref{arcSubtractionLBProp} holds for trip $r$, then consistency cut becomes
\begin{equation} \label{eq:case2CutStrongerRem}
\sum_{(h,l):z^1_{hl}=0} z_{hl} \geq \delta^{r}.
\end{equation}
\end{prop}

By leveraging the lifted network designs introduced in Section
\ref{liftingSection}, additional valid inequalities are proposed to
enhance the consistency cuts as follows.

\begin{prop}
For a given transit network design~$\bold{z}^1$, if the condition in Proposition \ref{arcAdditionLiftingThm} holds and the rider of trip $r$ adopts the ODMTS under~$\bold{z}^1$, then the consistency cut becomes: 
\begin{equation} \label{eq:StrongerConsistencyCutHubs}
\sum_{h \in \overrightarrow{W}_{or^r} (\bold{z}^1) \cup \overleftarrow{W}_{de^r} (\bold{z}^1), l \in H} z_{hl} + \sum_{(h,l):{\bold{z}^1}_{hl}=1} (1 - z_{hl}) + \delta^{r} \geq 1
\end{equation}
\end{prop}

\begin{proof}{Proof:}
For any design $\bold{z}^2$ in the form described in Proposition \ref{arcAdditionLiftingThm}, $t^2 \leq t^1$. Therefore, if the rider of trip $r$ adopts the ODMTS under~$\bold{z}^1$, then $t^2 \leq t^1 \leq \alpha^r t^r_{cur}$. This result implies adoption of the ODMTS for trip $r$ by setting $\delta^{r}$ to 1, under any design $\bold{z}^2$. $\square$
\end{proof}

For a given transit network design~$\bold{z}^1$, if the arc(s) satisfying the
condition in Proposition \ref{arcRemovalLiftingThm} are removed from 
$\bold{z}^1$, then the rider choices remain the same.

\begin{prop}
If the rider of trip $r$ adopts the ODMTS under design~$\bold{z}^1$, then the following inequality is valid:
\begin{equation} \label{eq:StrongerConsistencyCutArcs}
\sum_{h \in \mathcal{A}^r(\bold{z}^1)} (1 - z_{hl}) + \sum_{(h,l):{\bold{z}^1}_{hl}=0} z_{hl} +  \delta^{r} \geq 1 
\end{equation}
On the other hand, if the rider of trip $r$ does not adopt the ODMTS under~$\bold{z}^1$, then the following inequality is valid:
\begin{equation} 
\sum_{h \in \mathcal{A}^r(\bold{z}^1)} (1 - z_{hl}) + \sum_{(h,l):{\bold{z}^1}_{hl}=0} z_{hl}  \geq \delta^{r}
\end{equation}
\end{prop}

\subsection{Pareto-Optimal Cuts}
\label{sec:ParetoOptCuts}

To further accelerate the solution methodology, the decomposition
algorithm generates Pareto-optimal cuts \citep{Magnanti1981}.  Each
subproblem is first solved to identify its optimal objective function
value, i.e., $\Upsilon^r(\bar{z})$ for trip $r$ and design
$\bar{\bold{z}}$. The second step solves the Pareto subproblem 
\begin{subequations} \label{eq:dualLowerLevelPareto}
\begin{alignat}{1}
\max \quad & (u_{or^r}^r - u_{de^r}^r) - \sum_{h,l \in H} {z}^0_{hl} v^r_{hl} \label{eq:dualLowerLevelObjPareto} \\
\text{s.t.} \quad & u_h^r - u_l^r - v_{hl}^r \leq \hat{\tau}_{hl}^r \quad \forall h,l \in H \label{eq:dualLowerLevelConstr1Pareto} \\
& u_i^r - u_j^r \leq \hat{\gamma}_{ij}^r \quad \forall i,j \in A^r \label{eq:dualLowerLevelConstr2Pareto} \\
& (u_{or^r}^r - u_{de^r}^r) - \sum_{h,l \in H} \bar{z}_{hl} v^r_{hl} = \Upsilon^r (\bar{z}) \label{eq:dualLowerLevelConstr4Pareto} \\
& u_i^r \geq 0 \quad \forall i \in N, v_{hl}^r \geq 0 \quad \forall h,l \in H, \label{eq:dualLowerLevelConstr3Pareto}
\end{alignat}
\end{subequations}
where constraint \eqref{eq:dualLowerLevelConstr4Pareto} is added and
{\color{black}the objective function \eqref{eq:dualLowerLevelObjPareto}
  uses a core point $\bold{z}^0$ that satisfies the weak connectivity
  constraint \eqref{eq:upperLevelConstr1}. This core point can be
  selected from the relative interior of the convex hull of feasible
  network designs to obtain cuts that are not dominated by other
  optimality cuts. However, points that do not satisfy these criteria
  can be also used in the objective function to obtain valid cuts. In
  this study, for a given $\eta \in (0,1)$, this point is set as
  $z^0_{hl} = \eta$ for all $h,l \in H$. This selected point can be
  further updated through iterations to enhance the computational
  performance of this approach \citep{PAPADAKOS2009176}.}

\section{Computational Study}
\label{sec:ComputationalStudy}

This section presents a case study using a real data set from AAATA,
the transit agency serving the broader Ann Arbor and Ypsilanti area of
Michigan. Section \ref{sec:CaseStudySetting} introduces the
experimental setting. Section \ref{sec:ODMTSDesign} presents the ODMTS
design under different configurations, and provides a detailed
analysis in comparison to the current transit system. Section
\ref{compEfficiencySection} discusses the computational performance of
the proposed solution approach.

\subsection{Experimental Setting}
\label{sec:CaseStudySetting}

The case study is based on the AAATA transit system that operates over
1,267 bus stops, \textcolor{black}{in which 10 of these stops are
  designated as hubs in the baseline {\color{black}ODMTS} setting since they are located at
  high density corridors.}  It uses all the trips \textcolor{black}{utilizing the current transit system} from 6 pm to 10 pm,
i.e., which consists primarily \textcolor{black}{of} commuting trips from work to
home. There are 1,503 trips, each associated with an origin and a
destination bus stop, for a total of 5,792 riders {\color{black} as each trip can have multiple riders}. As the time and
distance between bus stop pairs do not satisfy triangular inequality,
a preprocessing step is applied to ensure this property. 

The experimental settings define different rider preferences depending
on income levels. More specifically, as the income level of the riders
increases, they become less tolerant to increases in trip duration. To
this end, the experiments categorize the trips into three groups:
high-income, middle-income, and low-income trips. This categorization
in income levels is based on the destination stop of each trip, which
is used as a proxy for the residential address of riders of that trip.
Out of the 1,503 trips, there are 476 low-income, 819 middle-income,
and 208 high-income trips with 1,754, 3,316, and \textcolor{black}{722} riders
respectively. The experimental settings also assume that all
low-income riders must use the transit system, whereas a certain
percentage of riders from middle-income and high-income levels have
the option to switch to the ODMTS from their current mode of travel
with personal vehicles. In particular, 100\%, 75\%, and 50\% of the
trips from the low-income, middle-income and high-income categories
must utilize the transit system, while the remaining ones have a mode
decision to make. Consequently, the value of the parameter $\alpha^r$
in choice function \eqref{eq:ChoiceFunctionTime} becomes smaller as the income level
of the riders increases. In particular, $\alpha^r$ is set to 1.5 and 2
for the trips associated with high-income and middle-income riders
respectively.

The bus cost per mile, $\rho$, is set to \$5.44 and the on-demand shuttle
cost per mile, $g$, is set to \$1.61. The price $\phi$ of using the
ODMTS \$2.50, which is in line with the fares of transit agencies. The
experimental setting assumes $n = 16$ buses within the four-hour
planning horizon for each open leg between the hubs with an average
waiting time $s$ of 7.5 minutes. The cost and inconvenience parameter
$\theta$ is 0.001 in the case study. {\color{black}As part of preprocessing, the shortest path between each node pair $i,j$ is precomputed based on the arc weights that are equal to the weighted cost and inconvenience of that pair if it is served by an on-demand shuttle, i.e. with the arc weights $(1 - \theta) g \, d'_{ij} + \theta t'_{ij}$, where $d'_{ij}$ and $t'_{ij}$ correspond to the distance and time metrics in the original data set. 
Using the resulting shortest path, the time $t_{ij}$ and distance $d_{ij}$ values between nodes $i,j$ are computed for each pair. 
Furthermore, the value of the parameter $\eta$ in Section \ref{sec:ParetoOptCuts} is set to 0.01 after comparing its computational performance against different values.}
Computational experiments are
conducted using Gurobi 9.0 as the solver on an Intel i5-3470T 2.90 GHz
machine with 8 GB RAM.

\subsection{Study of ODMTS Designs}
\label{sec:ODMTSDesign}

This section studies the ODMTS designs under different
assumptions. Section \ref{sec:TransitNetworkDesign} presents the
baseline ODMTS design and analyses its trip duration and adoption rates. The following sections examine how the baseline
design changes under various assumptions.  Sections
\ref{sec:IncreasedCostShuttles}--\ref{sec:IncreasedRidershipSectionandHubs}
{\color{black}examine} configurations where (1) the cost of operating on-demand
shuttles becomes higher, {\color{black} (2) travel inconvenience is penalized more}, (3) ridership increases, (4) travel choices
are associated with riders who cannot afford personal vehicles \textcolor{black}{for examining access to transit systems}, and
(5) the number of hubs is increased and the ridership also
grows. Finally, Section \ref{ref:CostAnalysisSection} compares the
baseline with the {\color{black} five} configurations with \textcolor{black}{respect} to adoption rates,
costs, and revenues obtained.

\subsubsection{The Baseline ODMTS Design}
\label{sec:TransitNetworkDesign}

The baseline ODMTS design is depicted in Figure
\ref{networkDesignFigure} and it opens 7 hubs. \textcolor{black}{Hub candidates are shown as black triangles and bus stops are colored
by income level: red dots in low-income regions, gray squares in middle-income
regions, and green pluses in high-income regions. 94\% of
middle-income and 74\% of high-income riders adopt the ODMTS.} 

\begin{figure}[!t]
\includegraphics[width=\textwidth]{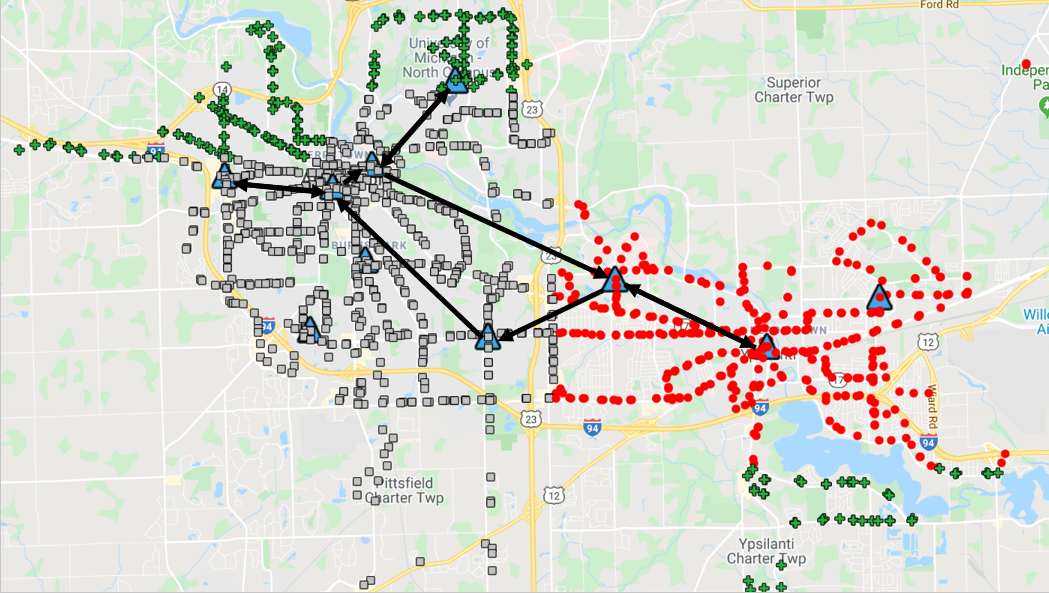}
\caption{Network Design for the ODMTS with 10 Hubs \textcolor{black}{(black triangles represent potential hub locations and black arrows show the open hub legs in the network design. In terms of bus stops; stops in low-income regions, middle-income regions, high-income regions correspond to red dots, gray squares, green pluses, respectively)}.}
\label{networkDesignFigure}
\end{figure}

\begin{table}[!t]
    \centering
   \begin{tabular}{rrrrrrrrrrrr}
   \hline
           & \multicolumn{3}{c}{{Riders adopting ODMTS}} & & \multicolumn{3}{c}{{Existing riders}} & & \multicolumn{3}{c}{{Riders not adopting ODMTS}} \\
           \cline{2-4} \cline{6-8} \cline{10-12}
{Income} & {ODMTS} & {direct} & {AAATA} & & {ODMTS} & {direct} & {AAATA} & & {ODMTS} & {direct} & {AAATA} \\
\hline
       low &             \multicolumn{ 3}{c}{NA} & &      16.05 &       6.90 &      25.63 &            & \multicolumn{ 3}{c}{NA} \\
    medium &       4.21 &       3.61 &      14.64 & &     11.27 &       5.03 &      21.53 &     & 25.91 &       7.73 &      31.88 \\
      high &       4.61 &       4.61 &      15.42 & &      9.84 &       5.31 &      21.06 &     & 19.96 &       8.37 &      29.77 \\
\hline
   \end{tabular}
\caption{Trip Duration Analysis under 10 Hubs Design.}
\label{TripDurationAnalysis}
\end{table}

\begin{table}[!t]
    \centering
   \begin{tabular}{rrrrrrrrrrrr}
   \hline
           & \multicolumn{3}{c}{{Riders adopting ODMTS}} & & \multicolumn{3}{c}{{Existing riders}} & & \multicolumn{3}{c}{{Riders not adopting ODMTS}} \\
           \cline{2-4} \cline{6-8} \cline{10-12}
{Income} & {ODMTS} & {direct} & {AAATA} & & {ODMTS} & {direct} & {AAATA} & & {ODMTS} & {direct} & {AAATA} \\
\hline
       low &             \multicolumn{ 3}{c}{NA} &  &    18.39 &       6.91 &      25.63 &            & \multicolumn{ 3}{c}{NA} \\
    medium &       3.21 &       2.82 &      12.19 & &     14.16 &       5.03 &      21.53 &     & 27.38 &       7.23 &      29.14 \\
      high &       4.47 &       4.47 &      14.42 & &     10.41 &       5.36 &      21.06 & &      21.09 &       8.37 &      29.99 \\
\hline
   \end{tabular}
    \caption{Trip Duration Analysis under 10 Hubs Design with Increased On-Demand Shuttle Cost.}
    \label{TripDurationAnalysis10HubsIncreasedShuttleCost}
\end{table}

Table~\ref{TripDurationAnalysis} reports various statistics on trip
durations per income level for existing riders, riders adopting the
designed ODMTS, and those not adopting it. More precisely, the table
uses the following classification: i) riders who
choose to adopt the ODMTS, ii) existing riders of the transit system
who have no mode choice and thus necessarily adopt the ODMTS, and iii)
riders with choice who do not adopt the designed ODMTS. {\color{black}For each rider
type and each income level, the table reports three  average trip durations over the corresponding rider sets:} the
duration in the designed ODMTS, the duration of the {\em direct} trip,
and the duration in the existing {\em AAATA} transit system.

The table highlights that the ODMTS routes are significantly shorter
than those of the existing transit system. For existing riders, the
trip durations reduced by 37\%, 48\%, and 53\% for low-income,
middle-income, and high-income riders. \textcolor{black}{This is
  critical since many of these riders may not have an alternative
  transportation mean, and the ODMTS should not increase the travel
  time for the vast majority of these riders.  In particular, out of
  1503 trips, 1347 trips utilize the ODMTS as either their riders
  prefer adopting the ODMTS or they are part of the existing
  trips. From the set of trips with choice who adopt the ODMTS, all
  trips have a travel time which is less than their corresponding
  travel time in the current transit system. On the other hand, a
  subset of the existing trips have longer trip
  durations. Specifically, out of 1347 trips, 11\% of trips (149
  trips) have longer travel time in the ODMTS with on average 7.99
  minutes longer trips. Note that this is a pessimistic estimate for
  the ODMTS as the transit times in the current system do not
  {\color{black}factor in} the time to walk from the true origin to the bus stop and
  from bus stop to the true destination, whereas the ODMTS picks up
  and drops off the riders (essentially) at their origin and to their
  destination. This result demonstrates that, for 89\% of the trips,
  ODMTS perform better compared to the current transit system with
  better convenience while being profitable at reasonable ticket
  prices as discussed in Section \ref{ref:CostAnalysisSection}.}

Furthermore, it is interesting to examine low-income riders whose
trips take longer than 40 minutes in the existing transit
system. These trips, called {\em low-income long transit (LILT)
  trips}, constitute 28\% of the low-income rides and have an average
transit time of 51.39 minutes. Under the baseline ODMTS design their
average trip duration decreased to 32.21 minutes, a 37\% reduction in
transit time.  For riders with mode choice, the durations of the
existing transit routes are also significantly reduced under the
baseline ODMTS design. Interestingly, riders who adopt the ODMTS have
routes almost as short as direct trips. The reduction in average trip
duration is 71\% and 70\% for middle-income and high-income riders who
adopt the ODMTS design, making the proposed ODMTS substantially more
attractive. The riders who do not adopt ODMTS have longer direct trip
times: although the baseline ODMTS improves over the existing system,
the reduction in transit time is not enough to induce a mode change.

\subsubsection{Impact of Increased Cost for On-Demand Shuttles}
\label{sec:IncreasedCostShuttles}

Consider the case where the cost of on-demand shuttles increases by
50\%.  Figure~\ref{networkDesignFigure10HubsIncreasedShuttleCost}
depicts the resulting ODMTS design which now opens all hubs and
significantly increases their connectivity. The resulting ODMTS thus
relies more on the bus network and less on the on-demand shuttles to
serve the trips. The overall adoption rates decreased slightly, as \textcolor{black}{92\% of the middle-income and 74\% of the high-income riders adopt the system. }
This reduction in adoption is obviously directly linked to
longer transit times. Table~\ref{TripDurationAnalysis10HubsIncreasedShuttleCost}
reports the average trip durations corresponding to each rider class
under this setting.

\begin{figure}[!t]
  \includegraphics[width=\textwidth]{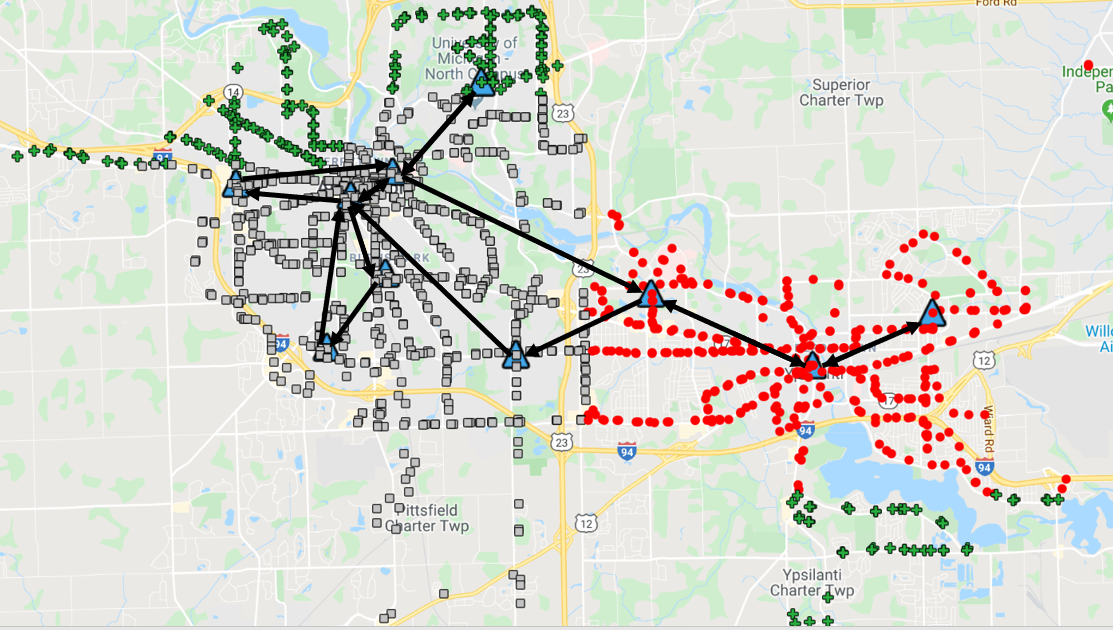}
\caption{Network Design for the ODMTS with 10 Hubs with Increased On-Demand Shuttle Cost.} \label{networkDesignFigure10HubsIncreasedShuttleCost}
\end{figure}

{\color{black}
\subsubsection{Impact of Weights of Cost and Inconvenience}
\label{sec:ImpactofThetaParameter}

This section studies the effect of the choice of the parameter $\theta$, which is used for adjusting the trade-off between cost and inconvenience in the weighted objective. It presents the results of the baseline instance in Section \ref{sec:TransitNetworkDesign} under a higher value of $\theta$ as 0.01, i.e., giving more weight to inconvenience and less weight to cost of the ODMTS. The resulting network design is shown in Figure~\ref{networkDesignFigure_IncreasedTheta}. Under this setting, in comparison to Figure~\ref{networkDesignFigure}, only three bus legs are open as the system aims at serving trips with shorter travel times, resulting in the usage of more on-demand shuttles. Table~\ref{TripDurationAnalysis_ThetaSmaller} summarizes the trip duration analysis under this setting, where 99\% of middle-income and 100\% of high-income riders adopt the ODMTS. As this ODMTS heavily depends on on-demand shuttles and do not benefit from the potential bus legs between hubs, it is not a desirable and sustainable system in comparison to the baseline setting with higher operational costs, as shown in Table
\ref{CostComparisonAnalysis}. As larger $\theta$ values give similar results, $\theta$ is selected as 0.001 throughout the computational study.}

\begin{figure}[!t]
\includegraphics[width=\textwidth]{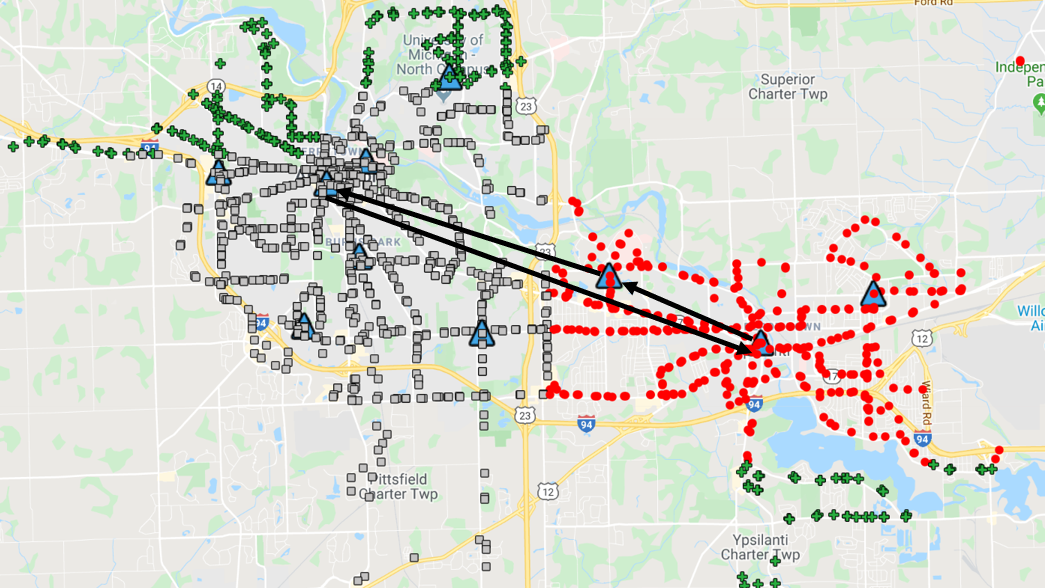}
\caption{\color{black} Network Design for the ODMTS with 10 Hubs with Increased $\theta$ parameter value.}
\label{networkDesignFigure_IncreasedTheta}
\end{figure}

\begin{table}[!t]
\color{black}
    \centering
   \begin{tabular}{rrrrrrrrrrrr}
   \hline
           & \multicolumn{3}{c}{{Riders adopting ODMTS}} & & \multicolumn{3}{c}{{Existing riders}} & & \multicolumn{3}{c}{{Riders not adopting ODMTS}} \\
           \cline{2-4} \cline{6-8} \cline{10-12}
{Income} & {ODMTS} & {direct} & {AAATA} & & {ODMTS} & {direct} & {AAATA} & & {ODMTS} & {direct} & {AAATA} \\
\hline
       low &             \multicolumn{ 3}{c}{NA} & &      8.47 &	6.70 &	25.63
 &            & \multicolumn{ 3}{c}{NA} \\
    medium &       5.75	& 5.16 &	21.65 &&	5.76	& 4.93 &	21.53
 &     & 31.99	& 15.03 &	70.52
 \\
      high &       6.98	& 6.79 &	24.69 &&	5.17	& 5.13 &	21.06
 &     & \multicolumn{ 3}{c}{NA} \\
\hline
   \end{tabular}
\caption{Trip Duration Analysis under 10 Hubs Design with $\theta = 0.01$.}
\label{TripDurationAnalysis_ThetaSmaller}
\end{table}

\subsubsection{Impact of Increased Ridership}
\label{sec:IncreasedRidershipSection}

This section examines the effect of increased ridership and studies
the ODMTS design when the number of riders doubles. The resulting
ODMTS design is illustrated in Figure
\ref{networkDesignFigure10HubsDoubledRidership}. Again, all of the
hubs are open and most of the bus legs from the baseline design also
operate in the new design. Furthermore, the design increases
connectivity to the lower-income communities by opening new bus legs
in the corresponding regions. On the other hand, adoption ratios in
terms of the trips decreased marginally: 
\textcolor{black}{92\% of middle-income and 74\% of high-income riders utilize the resulting system.} 

\begin{table}[!t]
    \centering
   \begin{tabular}{rrrrrrrrrrrr}
   \hline
           & \multicolumn{3}{c}{{Riders adopting ODMTS}} & & \multicolumn{3}{c}{{Existing riders}} & & \multicolumn{3}{c}{{Riders not adopting ODMTS}} \\
           \cline{2-4} \cline{6-8} \cline{10-12}
{Income} & {ODMTS} & {direct} & {AAATA} & & {ODMTS} & {direct} & {AAATA} & & {ODMTS} & {direct} & {AAATA} \\
\hline
       low &             \multicolumn{ 3}{c}{NA} & &     17.33 &       6.90 &      25.63 &  & \multicolumn{ 3}{c}{NA} \\
    medium &       3.71 &       3.17 &      13.69 &  &    12.06 &       5.03 &      21.53 &  & 24.71 &       7.30 &      29.31 \\
      high &       4.53 &       4.53 &      14.39 &   &   10.09 &       5.31 &      21.06 &  & 20.85 &       8.38 &      30.17 \\
\hline
   \end{tabular}
    \caption{Trip Duration Analysis under 10 Hubs Design with Doubled Ridership.}
    \label{TripDurationAnalysisDoubledRidership}
\end{table}

\begin{table}[!t]
    \centering
   \begin{tabular}{rrrrrrrrrrrr}
   \hline
           & \multicolumn{3}{c}{{Riders adopting ODMTS}} & & \multicolumn{3}{c}{{Existing riders}} & & \multicolumn{3}{c}{{Riders not adopting ODMTS}} \\
           \cline{2-4} \cline{6-8} \cline{10-12}
{Income} & {ODMTS} & {direct} & {AAATA} & & {ODMTS} & {direct} & {AAATA} & & {ODMTS} & {direct} & {AAATA} \\
\hline
       low &      32.40 &      11.99 &      51.50 & &      13.01 &       5.65 &      19.07 &    & 49.24 &      10.05 &      50.46 \\
    medium &       3.71 &       3.17 &      13.69 &  &    12.06 &       5.03 &      21.53 &     & 24.71 &       7.30 &      29.31 \\
      high &       4.53 &       4.53 &      14.39 &  &    10.09 &       5.31 &      21.06 &     & 20.85 &       8.38 &      30.17 \\
\hline
   \end{tabular}
    \caption{Trip Duration Analysis under 10 Hubs Design with Doubled Ridership and Rider Choices for LILT trips.}
    \label{TripDurationAnalysisDoubledRidershipLILTChoices}
\end{table}

\begin{figure}[!t]
  \includegraphics[width=\textwidth]{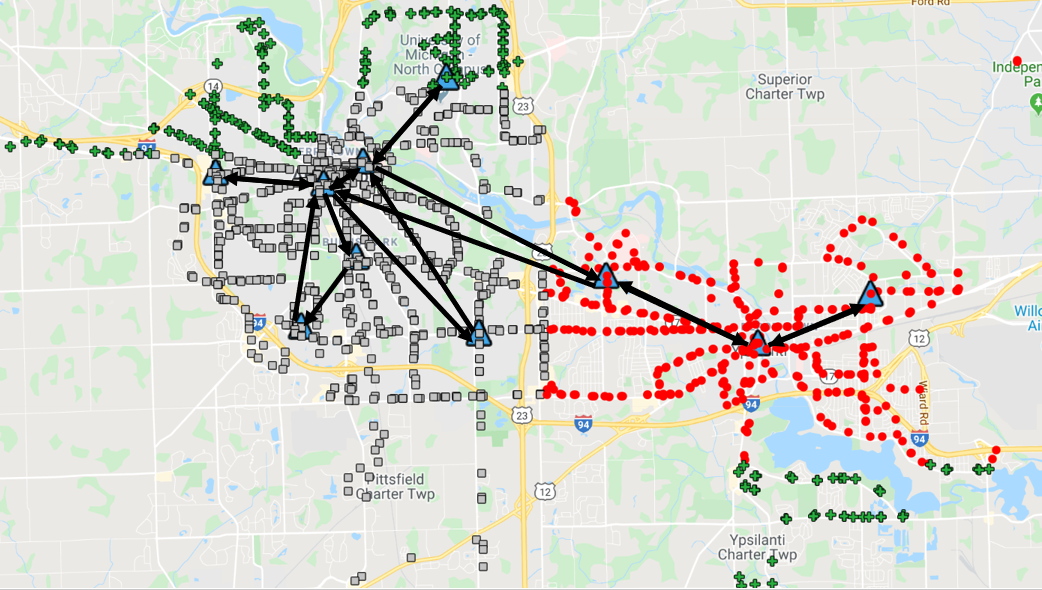}
  \caption{Network Design for the ODMTS with 10 Hubs with Doubled Ridership.} \label{networkDesignFigure10HubsDoubledRidership}
\end{figure}

Table \ref{TripDurationAnalysisDoubledRidership} presents the average
trip durations for this design. Similar to the base case, the ODMTS
performs better than the current transit system.  The trip durations
for existing riders become slightly longer in the new design as more
bus legs are utilized.

\subsubsection{Impact of {\color{black}Access Needs in ODMTS}}
\label{sec:IncreasedRidershipLowIncomeWithoutVehicles}

The next results concern \textcolor{black}{access to transit systems, a critical metric for transit agencies}. As mentioned earlier, it is critical to ensure that
low-income riders with no personal vehicles can be served by the transit
system within reasonable transit times. Otherwise, they may lose
\textcolor{black}{their access} to jobs, education, health-care, and other amenities,
since the trip duration may become impractical. Consider again the
LILT trips discussed in Section \ref{sec:TransitNetworkDesign}. To
study \textcolor{black}{these access needs to transit systems}, these trip riders are associated with a choice
model with $\alpha^r$ parameter set to 4. If a trip duration becomes
longer than four times than the direct trip time, these riders will
not \textcolor{black}{be} able to utilize the system anymore and lose \textcolor{black}{access} to major
opportunities. Out of 476 low-income trips, there are 132 such LILT
trips. The results are presented for the case of doubled ridership.

Under this model, 96\% of low-income trips utilize the ODMTS system
and almost all of the LILT riders adopt the ODMTS, demonstrating the
system ability to meet \textcolor{black}{access} needs. The ODMTS design is the
same as in Figure~\ref{networkDesignFigure10HubsDoubledRidership}.

Table~\ref{TripDurationAnalysisDoubledRidershipLILTChoices} presents
the trip duration results with this choice model and doubled
ridership. As the design remains the same, the middle-income and
high-income trips have the same adoption rates and trip durations as
in Table~\ref{TripDurationAnalysisDoubledRidership}.  LILT riders who
adopt the ODMTS have an average trip duration less than three times
that of the direct trip duration, and significantly shorter than the
average trip duration by the existing transit system. On the other
hand, LILT riders who do not adopt the ODMTS have much longer trip
durations, although they have shorter trips on average compared to the
current
system. Figure~\ref{networkDesignFigure10HubsDoubledRidership_LILTNotAdopting}
visualizes two of them, which are representative of trips for which
riders do not adopt the ODMTS. The trips share the same destination
(denoted by ``de'') but have different origins (denoted by ``or1'' and
``or2'').  Their routes are illustrated with orange dashed routes from
origins to destination. More specifically, the trip with origin
``or1'' uses an on-demand shuttle to reach the closest open hubs, but
results in a long trip due to many transfers between hubs. On the
other hand, the trip with origin ``or2'' utilizes the on-demand
shuttles for longer trip segments, but it involves a transfer to the
city center, increasing the trip duration. In general, however, all
the LILT trips with destination points in the vicinity of the
eastern-most hub adopt the ODMTS even when their origins are in the
city center.

\begin{figure}[!t]
  \includegraphics[width=\textwidth]{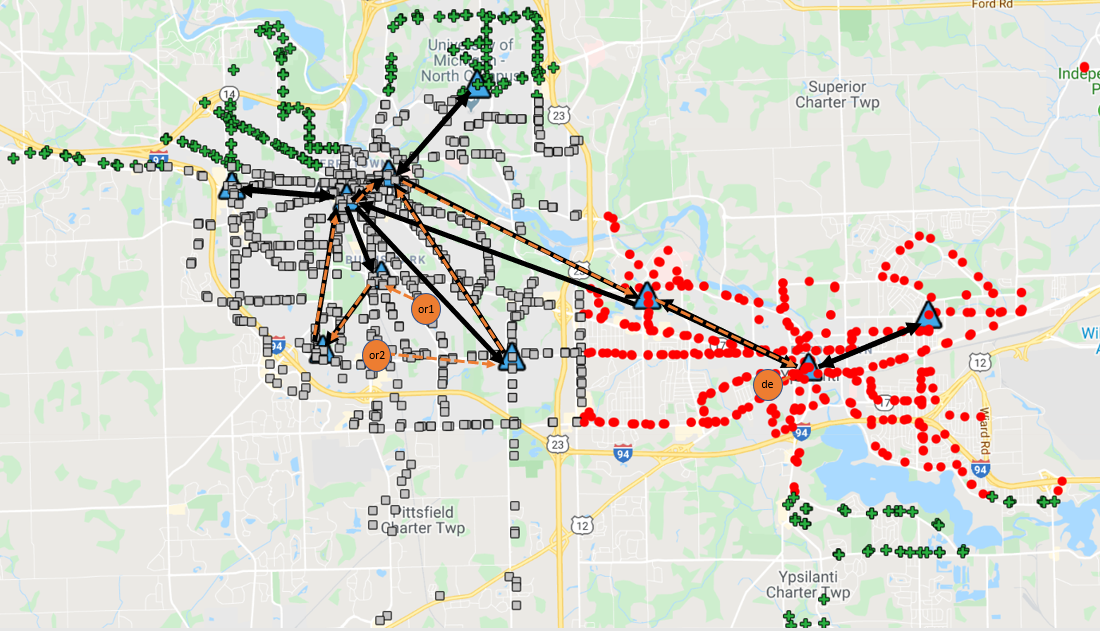}
  \caption{Visualization of Sample LILT Trips Not Adopting ODMTS.} \label{networkDesignFigure10HubsDoubledRidership_LILTNotAdopting}
\end{figure}

\subsubsection{Impact of Number of Hubs}
\label{sec:IncreasedRidershipSectionandHubs}

It is also interesting to study the effect of increasing the number of
hubs as ridership increases. Figure
\ref{networkDesignFigure10HubsDoubledRidership20Hubs} presents the
ODMTS design for 20 hubs and doubled ridership. The resulting design
opens 14 hubs and the bus network has a significantly broader
geographical coverage.  In this setting, \textcolor{black}{91\% of middle-income and 73\% of high-income riders adopt the ODMTS respectively.} 
Table
\ref{TripDurationAnalysisDoubledRidership20Hubs} reports the average
trip duration: the more expansive bus network induces increases of
11\%, 18\%, 1\% in average trip durations for low-income,
middle-income, and high-income riders respectively.
\textcolor{black}{Additionally, for the LILT trips, their average trip duration reduced from 51.39 minutes in the current transit system to 36.74 minutes in this setting, which is a 29\% decrease on trip duration despite of having on average 2.5 minutes longer trips than the analogous ODMTS design for 10 hubs.}

\begin{figure}[!t]
  \includegraphics[width=\textwidth]{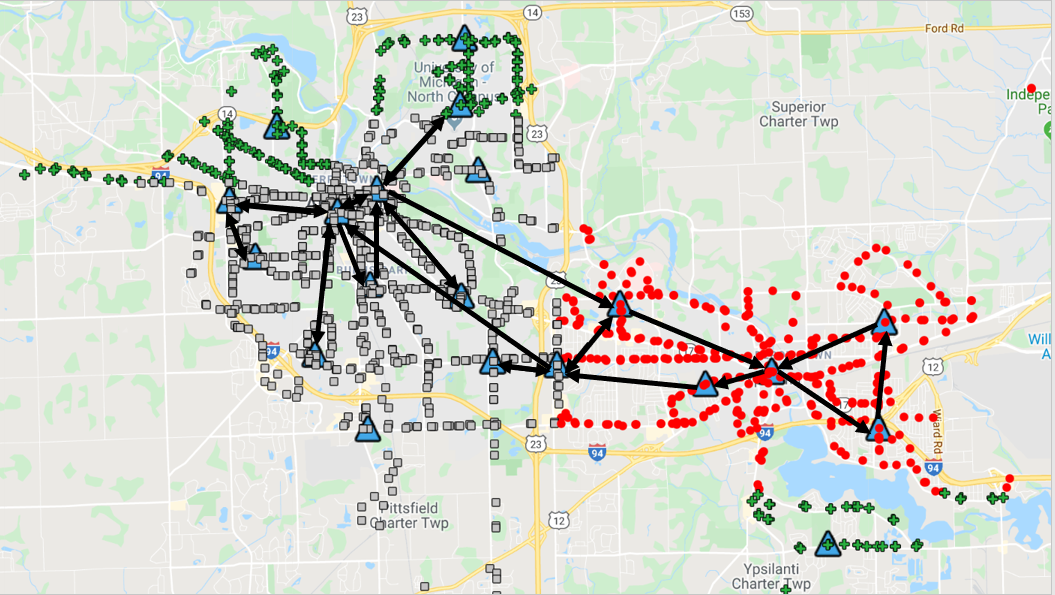}
  \caption{Network Design for the ODMTS with 20 Hubs with Doubled Ridership.} \label{networkDesignFigure10HubsDoubledRidership20Hubs}
\end{figure}

\begin{table}[!t]
    \centering
   \begin{tabular}{rrrrrrrrrrrr}
   \hline
           & \multicolumn{3}{c}{{Riders adopting ODMTS}} & & \multicolumn{3}{c}{{Existing riders}} & & \multicolumn{3}{c}{{Riders not adopting ODMTS}} \\
           \cline{2-4} \cline{6-8} \cline{10-12}
{Income} & {ODMTS} & {direct} & {AAATA} & & {ODMTS} & {direct} & {AAATA} & & {ODMTS} & {direct} & {AAATA} \\
\hline
       low &             \multicolumn{ 3}{c}{NA} & &      19.21 &       6.90 &      25.63 & &             \multicolumn{ 3}{c}{NA} \\
    medium &       3.05 &       2.64 &      11.22 & &     14.19 &       5.03 &      21.53 & &     24.21 &       7.12 &      28.94 \\
      high &       4.02 &       4.02 &      14.02 & &     10.17 &       5.31 &      21.06 & &      20.26 &       8.41 &      29.54 \\
\hline
   \end{tabular}
    \caption{Trip Duration Analysis under 20 Hubs Design with Doubled Ridership.}
    \label{TripDurationAnalysisDoubledRidership20Hubs}
\end{table}

\subsubsection{Adoption and Cost Analysis}
\label{ref:CostAnalysisSection}

{\color{black}Tables \ref{AdoptionComparisonAnalysis} and \ref{CostComparisonAnalysis} present a detailed comparison of
the ODMTS designs considered in Sections \ref{sec:TransitNetworkDesign}-\ref{sec:IncreasedRidershipSectionandHubs} with respect to the
adoption, cost, and revenue.} The revenue is assumed to be \$2.5 per
ride. {\tt 10Hub} refers to the baseline design from Section
\ref{sec:TransitNetworkDesign}, {\tt 10HubISC} to the 10 hub design
with increased on-demand shuttle costs from
Section~\ref{sec:IncreasedCostShuttles}, 
{\color{black} {\tt 10HubMWI} to the 10 hub design with more weight to minimizing inconvenience,} 
{\tt 10HubDR} to the 10 hub
design with doubled ridership from Section
\ref{sec:IncreasedRidershipSection}, {\tt 10HubDRAC} to the 10 hub
design with doubled ridership and considerations \textcolor{black}{of access} from
Section~\ref{sec:IncreasedRidershipLowIncomeWithoutVehicles}, and {\tt
  20HubDR} to the 20 hub design with doubled ridership from Section
\ref{sec:IncreasedRidershipSectionandHubs}.  \textcolor{black}{In Table \ref{AdoptionComparisonAnalysis},} columns ``MI'' and ``HI'' under ``Adoption (\%)'' column represent the percentage of the middle and high income riders who
adopt the ODMTS. No low-income riders have a choice model, except in
{\tt 10HubDRAC} \textcolor{black}{where 3428 of 3508 low-income riders adopt the ODMTS. Column ``\# of Riders'' corresponds to the number of
riders considered in the design, with the number of riders utilizing
the ODMTS in parentheses for middle-income, high-income and total riders, respectively}.  \textcolor{black}{In Table \ref{CostComparisonAnalysis},} columns ``Revenue'', ``Inv Cost'', and
``Trv Cost'' represent the revenue of the transit agency (from existing users and those
choosing to adopt the ODMTS), the investment cost of operating bus
legs between hubs, and the total travel cost of the ODMTS riders.
Column ``Net Cost/Rider'' presents the cost (or benefit) per rider: it is
obtained by deducting the revenue from the sum of the investment and
travel costs and dividing by the number of ODMTS riders.

\begin{table}[!t]
\centering
{\color{black}
\begin{tabular}{rc@{}c@{}cccc}
\hline
& \multicolumn{2}{c}{Adoption (\%)} && \multicolumn{3}{c}{\# of Riders} \\
\cline{2-3} \cline{5-7} 
           & MI & HI  && MI & HI & Total \\
\hline
     {\tt 10Hub} &   94 & 74  &&  3316 (3112)
 & 722 (536) & 5792 (5402) \\
     {\tt 10HubISC} & 92 & 74 && 3316 (3040)
 & 722 (532) & 5792 (5326) \\
 {\color{black}{\tt 10HubMWI}} & {\color{black}99} & {\color{black}100} && {\color{black}3316 (3312)}
 & {\color{black}722 (722)} & {\color{black}5792 (5788)} \\
   {\tt 10HubDR} &    92 & 74 &&  6632 (6124)
 & 1444	(1068) & 11584 (10700) \\
   {\tt 10HubDRAC} & 92 & 74 && 6632 (6124)
 & 1444	(1068) & 11584 (10620) \\
   {\tt 20HubDR} &   91 & 73 &&   6632 (6052)
 & 1444	(1048) & 11584 (10608) \\
\hline
\end{tabular}}
\caption{Adoption Comparison under Different ODMTS Settings.}
\label{AdoptionComparisonAnalysis}
\end{table}

\begin{table}[!t]
\centering
{\color{black}
\begin{tabular}{rcrrrr}
\hline
& \multicolumn{4}{c}{Revenue \& Costs} \\
\cline{2-5}
           &    Revenue &   Inv Cost &   Trv Cost & Net Cost/Rider \\
\hline
     {\tt 10Hub} &      13505.00 &    2440.80 &   13553.31 &       0.46 \\
     {\tt 10HubISC} & 13315.00 & 3564.59 & 17516.07 & 1.46 \\
      {\color{black}{\tt 10HubMWI}} & {\color{black}14470.00} & {\color{black}1429.86} & {\color{black}22153.77}
 & {\color{black}1.57} \\
   {\tt 10HubDR} &      26750.00 &    4073.14 &   23847.84 &       0.11 \\
   {\tt 10HubDRAC} & 26550.00 & 4073.14 & 23642.55 & 0.11 \\
   {\tt 20HubDR} &      26520.00 &    4959.34 &   20285.19 &      -0.12 \\
\hline
\end{tabular}}
\caption{Cost and Revenue Comparison under Different ODMTS Settings.}
\label{CostComparisonAnalysis}
\end{table}

The first interesting result is that the baseline design would be
profitable for a price of \$2.96, which is quite remarkable, given the
improvements in quality of service and the increased ridership. Of
course, the analysis ignores a variety of fixed costs and subsidies
but the analysis reflects the significant ODMTS potential.  As
ridership grows, revenues also grow in proportion and the adoption
rates remain similar. The investment cost for the bus network and the
travel costs of the on-demand shuttles also grow but slower: this
means that the net cost per rider decrease significantly, highlighting
economies of scale in ODMTS. The 20-hub design is particularly
interesting: the investment cost for the buses further increases but
the cost for on-demand shuttles decreases more, making the ODMTS
profitable at \$2.5.

Capturing travel mode adoption in the design of ODMTS ensures that the
transit system will be sized properly and have the targeted level of
performance. However, it is also interesting to mention the financial
benefits of modeling mode adoption. By scaling the obtained results
for 52 weeks, 5 days a week, and 12 hours a day, the bilevel
optimization model would produce savings of \$165,937, \$302,350, and
\$120,631 for {\tt 10HubDR}, {\tt 20HubDR}, and {\tt 10HubISC}
respectively. 

\subsection{Computational Efficiency}
\label{compEfficiencySection}

\begin{table}[!t]
\centering
\begin{tabular}{rrrrrrr}
\hline
     &      & \multicolumn{2}{c}{10 hubs} & & \multicolumn{2}{c}{20 hubs} \\
\cline{3-4} \cline{6-7}
{Income} & {\# of trips} & {direct trips} &  {\% Identified} & & {direct trips} &  {\% Identified} \\
\hline
 {low} &        476 &        145 &      30.46 & &       106 &      22.27 \\
{medium} &        819 &        260 &      31.75 & &       220 &      26.86 \\
{high} &        208 &         80 &      38.46 &  &       52 &      25.00 \\
\hline
{Total} &       1503 &        485 &      32.27 &  &      378 &      25.15 \\
\hline
\end{tabular}
\caption{Direct Trip Identification Analysis.}
\label{directTripIdentificationTable}
\end{table}

\begin{figure}[!t]
\centering
\begin{subfigure}[b]{.45\textwidth}
  \centering
  \caption{10 Hubs Instance.}
\label{compEfficiencyFigure_10Hubs}
\begin{tikzpicture}[scale=0.85]
\begin{axis}[
	xlabel={Run time ($10^3$ seconds)},
	ylabel={Optimality gap (\%)},
	xmin = 0, xmax = 3.6,
	xtick = {0, 1, 2, 3},
	ymin = 0, ymax = 100,
	ytick = {0,10,20,30,40,50,60,70,80,90,100},
     legend pos=north east,
]
\addplot table {DefaultRunRescaled.dat};

\addplot table {EnhancedRunRescaled.dat};

\legend{Base case, Enhanced case}
\end{axis}
\end{tikzpicture}
\end{subfigure}
\begin{subfigure}[b]{.45\textwidth}
  \centering
  \caption{20 Hubs Instance.}
 \label{compEfficiencyFigure_20Hubs}
 \begin{tikzpicture}[scale=0.85]
\begin{axis}[
	xlabel={Run time ($10^4$ seconds)},
	ylabel={Optimality gap (\%)},
	xmin = 0, xmax = 4.3,
	xtick = {0, 1, 2, 3},
	ymin = 0, ymax = 100,
	ytick = {0,10,20,30,40,50,60,70,80,90,100},
     legend pos=north east,
]
\addplot table {DefaultRun_20HubsRescaled.dat};

\addplot table {EnhancedRun_20HubsRescaled.dat};

\legend{Base case, Enhanced case}
\end{axis}
\end{tikzpicture}
\end{subfigure}
\caption{Impact of the Enhancements on Computational Performance.}
\label{compEfficiencyFigure}
\end{figure}
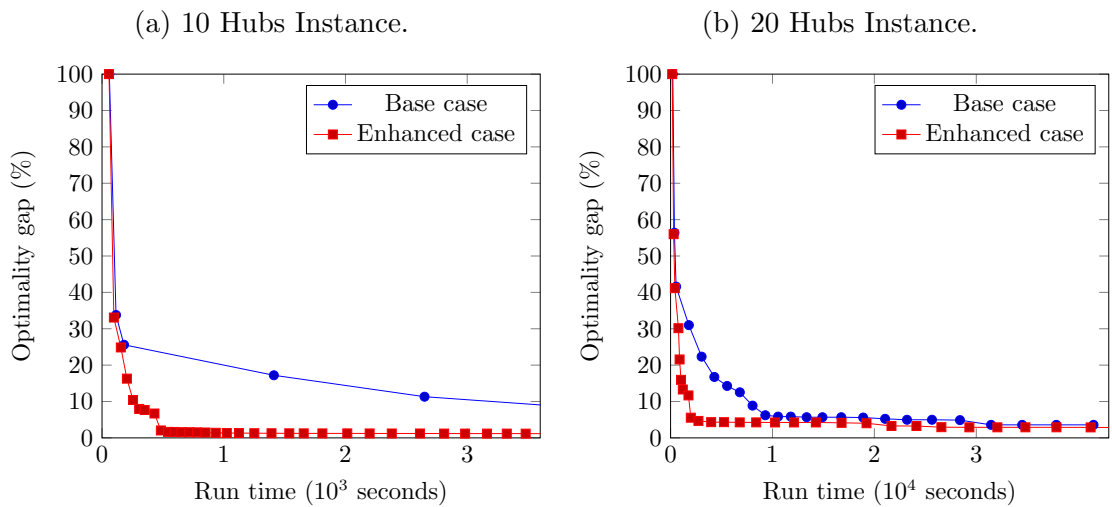

This section reports a number of computational results on the bilevel
optimization model, including the impact of the preprocessing steps
and the valid inequalities. Table~\ref{directTripIdentificationTable}
reports on the ability to detect direct trips for instances with 10
and 20 hubs.  32\% and 25\% of the trips are identified as direct in
the 10 hubs and 20 hubs instances. The percentage decreases for 20
hubs since the bus network is more expansive. In 10 hubs setting, the
highest percentage of direct trips are high-income, as the hub
locations are further away from the origin and destination of these
trips. This percentage reduces substantially for 20 hubs for
high-income class, especially in comparison to other rider classes,
demonstrating the importance of hub locations and the number of hubs
for this analysis.

Figures \ref{compEfficiencyFigure_10Hubs} and
\ref{compEfficiencyFigure_20Hubs} examine the benefits of the bounds on the follower problem presented in
Section~\ref{sec:UBonFollowerProblemObj} in combination with the valid
inequalities proposed in Section~\ref{sec:ValidIneq} \textcolor{black}{in comparison to a standard Benders decomposition algorithm with the nogood cuts for ensuring consistency between rider choices and network designs by excluding these enhancements}. The figures use
the baseline instance with 10 Hubs studied in
Section~\ref{sec:TransitNetworkDesign} and the 20 Hubs instance
studied in Section~\ref{sec:IncreasedRidershipSectionandHubs}. {\color{black} They
report the optimality gap and the run time at each iteration of the
algorithm over a time limit of 1 hour and 10 hours for the 10 Hubs and 20 Hubs instances, which terminate with optimality gaps of \%8.40 and \%3.58 for the base case and \%1.15 and \%2.84 for the the enhanced case, respectively. Furthermore, under the same time limit, the base case and enhanced case are able to conduct 6 and 33 iterations for the 10 Hubs instance, and 22 and 26 iterations for the 20 Hubs instance. Note that 10 Hubs instance can be solved to optimality within 2 hours.    
The results demonstrate the significant computational impact of the bounds and valid
inequalities: the proposed decomposition algorithm is capable of producing high-quality solutions in a reasonable amount of time for this real case study and
brings improvements of several orders of magnitude compared to a decomposition algorithm that does not preprocess trips with respect to the bounds and only relies on Benders and classical nogood cuts.}

As doubling the ridership in the case study considers the same origin-destination pairs with increased ridership amounts, the computational performance is not impacted by this change. {\color{black}On the other hand, increasing the number of distinct origin-destination pairs will typically impact the run time required for convergence of the solution algorithm.
Table \ref{RunTimeComparisonDifferentTripSizes} highlights these results over instances with different trip sizes, which are randomly selected from the set of trips. It compares the runtimes of the algorithm when riders are all adopting the transit (no latent trips) and when some riders may adopt the system (latent trips) depending on the mode choice model, as  discussed in Section~\ref{sec:CaseStudySetting}. The results show how 
much more challenging the problem becomes when latent demand is taken into account. However, the algorithms presented in this paper, are still capable of addressing this planning problem and obtain small optimality gaps. This is significant, since latent demand is a key worry of transit agencies as mentioned in the introduction. Obviously, improved computational methods are an important topic for future research.}

\begin{table}[!t]
\centering
\begin{tabular}{ccrrrrr}
\hline
     &      & \multicolumn{2}{c}{10 hubs} & & \multicolumn{2}{c}{20 hubs} \\
\cline{3-4} \cline{6-7}
{\# of trips} & {latent trips} & {\# of iter.} &  {run time (s)} &&  {\# of iter.} &  {run time (s)} \\
\hline
\multirow{2}{*}{100} &    No     &   2	& 4.96 &&	2	& 9.91 \\
 &   Yes      &      2	& 6.76	&& 2	& 11.73  \\
  \multirow{2}{*}{200} &     No    &        2 &	8.90 &&	2&	19.13
       \\
     &     Yes    &        3&	14.12&&	3&	24.66
   \\
      \multirow{2}{*}{500} &     No    &       5&	50.58 &&	5	&94.96
     \\
         &    Yes    &      88&	1316.76	&&94	&18000.00*
   \\
\hline
\end{tabular}
\caption{Run time comparison over instances with different trip sizes (*This instance reached the time limit with an optimality gap of \%2.0).}
\label{RunTimeComparisonDifferentTripSizes}
\end{table}

\section{Conclusion}
\label{sec:Conclusion}

This paper studied how to integrate rider mode preferences into the
design of ODMTS. This functionality was motivated by the desire to
capture the impact of latent demand, a key worry of transit
agencies. The paper proposed a bilevel optimization model to address
this challenge, in which the leader problem determines the ODMTS
design, and the follower problems identify the most cost efficient and
convenient route for riders under the chosen design. The leader model
contains a choice model for every potential rider that determines
whether the rider adopts the ODMTS given her proposed route. 

To solve the bilevel optimization model, the paper proposed a
decomposition method that includes Benders optimal cuts and nogood
cuts to ensure the consistency of the rider choices in the leader and
follower problems. Moreover, to improve the computational efficiency
of the method, the paper proposed upper and lower bounds on trip
durations for the follower problems and valid inequalities that
strenghten the nogood cuts using the problem structure.

The paper also presented an extensive computational study on a real
data set from AAATA, the transit agency for the broader Ann Arbor and
Ypsilanti region in Michigan. The study considered the impact of a
number of factors, including the price of on-demand shuttles, the
number of hubs, and \textcolor{black}{access to transit systems} criteria. It analyzed the adoption
rate of the ODTMS for various class of riders (low-income,
middle-income, and high-income). The designed ODMTS feature high
adoption rates and significantly shorter trip durations compared to
the existing transit system both for existing riders and riders who
adopted the ODMTS. Under increased ridership and/or the availability
of more hubs, trip durations may increase as they use more bus legs
between hubs and less on-demand shuttles; however, adoption rates are
not impacted much and the net profit of the transit agency increases
significantly through economies of scale. The results further
highlighted the benefits in \textcolor{black}{ensuring access} for low-income riders as
their \textcolor{black}{trip} durations decrease and remain reasonable.  Finally, the
computational study demonstrated the efficiency of the decomposition
method for the case study and the benefits of computational
enhancements.

Future work will consider more complex choice models (e.g., involving
the increasing cost of transfers or \textcolor{black}{probabilistic choice functions \citep{Paneque2021_ChoiceModels}}) and/or restrictions on acceptable
routes. Scaling the approach to large metropolitan areas is also a
priority.

\ACKNOWLEDGMENT{Many thanks to Julia Roberts at AAATA for sharing the
  transit data and for many interesting discussions. This research is
  partly supported by NSF Leap HI proposal NSF-1854684, and Department
  of Energy Research Award 7F-30154.}

\begin{APPENDIX}{}

{\color{black}
\section{Comparison with the Single-level Formulation}
\label{sec:ComparisontoSingleLevel}

This section presents a single-level formulation for the bilevel problem in Figure~\ref{fig:bilevel} to demonstrate the need to adopt a bilevel approach. Figure \ref{fig:singleLevel} presents the single-level problem. which moves the constraints of the lower level problem to the upper level. For simplicity, the lexicographic objective in the follower problem is omitted. 
\begin{figure}[h]
\begin{subequations} \label{eq:SingleLevelProblem}
\begin{alignat}{1}
\min_{z_{hl}, x^r_{hl}, y^r_{hl}, b^r, f^r, \delta^r} \quad & \sum_{h,l \in H} \beta_{hl} z_{hl} + \sum_{r \in T \setminus T'} p^r b^r + \sum_{r \in T'} p^r \delta^r (b^r - \varphi) \\
\text{s.t.} \quad & \sum_{l \in H} z_{hl} = \sum_{l \in H} z_{lh} \quad \forall h \in H  \\
   & b^r = \sum_{h,l \in H} \tau_{hl}^r x_{hl}^r + \sum_{i,j \in N} \gamma_{ij}^r y_{ij}^r \quad \forall r \in T \\
   & f^r = \sum_{h,l \in H}  (t_{hl} + s) x_{hl}^r + \sum_{i,j \in N}  t_{ij} y_{ij}^r  \quad \forall r \in T \\
   & \sum_{\substack{h \in H \\ \text{if } i \in H}} (x_{ih}^r - x_{hi}^r) + \sum_{j \in N}  (y_{ij}^r - y_{ji}^r) = \begin{cases}
     1 &, \text{if  } i = or^r \\
    -1 &, \text{if  } i = de^r \\
    0 &, \text{otherwise}
    \end{cases} \quad \forall i \in N, \> \forall r \in T \\
  & x_{hl}^r \leq z_{hl} \quad \forall h,l \in H, \quad \forall r \in T  \\
  & \delta^r = {\cal C}^r(\bold{x^r}, \bold{y^r}) \quad \forall r \in T' \label{eq:ChoiceFnc_SingleLevel} \\
& z_{hl} \in \{0,1\} \quad \forall h,l \in H, \quad \delta^r \in \{0,1\} \quad \forall r \in T'  \\
  & x_{hl}^r \in \{0,1\} \quad \forall h,l \in H, \quad y_{ij}^r \in \{0,1\} \quad \forall i,j \in N. 
\end{alignat}
\end{subequations}
\caption{\textcolor{black}{The Single-level Optimization Model for ODMTS Design with Travel Mode Adoption.}}
\label{fig:singleLevel}
\end{figure}

The choice function of every trip $r$ depends on the trip durations $f^r$ as defined in \eqref{eq:ChoiceFunctionTime}. To represent this relationship, constraint \eqref{eq:ChoiceFnc_SingleLevel} can be linearized as follows: 
\begin{align*}
f^r & \geq \alpha^r t^r_{cur} + \epsilon_f - M_f \delta^r, \\
f^r & \leq \alpha^r t^r_{cur} + M_f (1 - \delta^r),
\end{align*}
where $\epsilon_f \approx 0$ and $M_f$ is an upper bound on all of the trip durations under any network design. 

This formulation only evaluates the suggested routes and choices of the riders from the perspective of the transit agency, who consequently can suggest longer routes to the riders with choice if serving them is not profitable. Thus, their inconvenience is explicitly omitted in the system, which is undesirable for ensuring the access to the transit system. 

To illustrate this potential behavior, this section presents a numerical study over the provided baseline setting in  Section~\ref{sec:TransitNetworkDesign}. The instance is built by randomly selecting 100 trips from the data set. For giving more riders the choice of adoption in this setting, all trips from low-income riders are considered as existing riders, whereas all trips from middle income and high income riders constitute the latent demand. Table \ref{SingleLevelModelComparison} summarizes the comparison of the solutions of the bilevel problem in Figure~\ref{fig:bilevel} and the single-level problem in Figure~\ref{fig:singleLevel} in terms of rider adoption. Since the single-level problem is a relaxation to the bilevel problem, it results in a smaller objective function value. However, the single-level problem has a much lower adoption for all riders with choice and it explicitly suggests  longer routes to certain riders because serving them is not of direct benefit to the transit agency in terms of the objective function. This artificial removal of riders from the transit system also results in a different design with fewer opened bus legs. 

These results highlight the need for the bi-level model in Figure~\ref{fig:bilevel} in order to eliminate this pathological and unfair behavior. This is aligned with the objectives of many transit agencies which aims at using ODMTS to improve mobility for underserved communities.



\begin{table}[h]
    \centering
   \begin{tabular}{rrrr}
   \hline
    & \# of trips (\# of existing trips) &  \multicolumn{2}{c}{\# of trips adopting ODMTS} \\
    \cline{2-4}
           &&  Single-level & Bilevel \\
\hline
       low &   34 (34) & 34 & 34 \\
    medium &  58 (0) & 17 & 41   \\
      high &  8 (0) & 2 & 6   \\
\hline
   \end{tabular}
\caption{Single-level and Bilevel models comparison over a sample instance.}
\label{SingleLevelModelComparison}
\end{table}
}

{\color{black}
\section{Comparison with the Former Studies}
\label{sec:FinalComparisonCPAIOR}

This section expands the discussion presented in Section \ref{sec:DiscussiontoCPAIORPaper} to compare this study with the former study \citep{Basciftci2020} in terms of the novel analytical results derived in Section \ref{SectionAnalyticalResultsonTripDurations}, solution algorithm presented in Section \ref{sec:solutionMethodologySection} and case studies in Section \ref{sec:ComputationalStudy}. As the former paper studies an aligned choice model with the objective of the follower problem $b^r$, it benefits from the following result: Since the follower problem obtains the shortest path from origin to destination of a given trip with respect to the weighted cost and convenience of the arcs, $b^r$ value decreases as more hub legs become available. Then, the paper benefits from anti-monotone choice functions that are defined as follows. 
\begin{definition}[Anti-Monotone Mode Choice] A choice function ${\cal C}^r$ is anti-monotone if
$
  b^r_1 \leq b^r_2 \Rightarrow {\cal C}^r(b^r_1) \geq {\cal C}^r (b^r_2).
$
\end{definition}
Observe that the choice function ${\cal C}^r(b^r) \equiv \mathbbm{1} (b^r \leq \alpha^r \ b^r_{cur})$ is anti-monotone since $b_1^r \leq b_2^r$ implies ${\cal C}^r(b^r_1) \geq {\cal C}^r(b^r_2)$. Thus, to obtain the case $b_1^r \leq b_2^r$, we can simply consider evaluating $b^r$ under the designs $\bold{z}^1 \geq \bold{z}^2$, where $b_i^r$ represents $b^r$ value under design $\bold{z}^i$. Under these relationships, nogood cuts \eqref{eq:case2Cut} and \eqref{eq:case3Cut} to ensure consistency between rider choices and design variables can be strengthened directly to the ones in
\eqref{eq:case2CutStronger} and \eqref{eq:case2CutStrongerRem} by adding or removing arcs from a given design, respectively, without deriving any further conditions. Having aligned objectives between the follower problem and the choice function along with the stronger cuts result in the fast convergence of the Benders decomposition based solution algorithm. 

On the other hand, although the choice function studied in this paper
is anti-monotone in terms of $f^r$, there is no direct relationship
between the network design variable $z$ and the convenience $f^r$ as
opening or closing of hub legs does not necessarily improve or
deteriorate the convenience of the trips. Thus, the former results do
not apply and these not aligned objectives complicate the solution
procedure. To be able to strengthen the consistency cuts from nogood
cuts in this setting, further analytical results are derived in
Section \ref{SectionAnalyticalResultsonTripDurations}. This analysis
provides sufficient conditions to obtain the stronger cuts
\eqref{eq:case2CutStronger} and
\eqref{eq:case2CutStrongerRem}. Furthermore, to accelerate the
solution algorithm, the problem size is reduced by identifying the
direct trips derived through these analyses, as demonstrated in Table
\ref{directTripIdentificationTable} over the studied
instances. Moreover, stronger cuts in the form of
\eqref{eq:StrongerConsistencyCutHubs} and
\eqref{eq:StrongerConsistencyCutArcs} are derived by identifying
certain hub legs whose addition or removal from a given design will
not impact the convenience and consequently the adoption behavior of
the riders.  Furthermore, upper and lower bounds on the follower
problem are presented in Section~\ref{sec:UBonFollowerProblemObj} to
strengthen the presented formulation.  The experiments demonstrate the
significant computational benefits of adopting the proposed
enhancements under this complicating setting with not aligned
objectives.

In addition to the differences in the modelling perspectives discussed in Section \ref{sec:DiscussiontoCPAIORPaper} and these novel technical results tailored for this problem setting, this paper provides an extensive case study over the broader Ann Arbor and Ypsilanti area of
Michigan over various instance settings. For each instance, the average trip time of each rider class depending on their adoption behaviour and income level are presented in comparison to the current transit system and direct travel option. The case study further presents results under different numbers of hubs, initial ridership amounts, on-demand shuttle costs, and with additional \textcolor{black}{concerns on access to transit systems}. These results demonstrate the performance of the ODMTS with high adoption percentages and better convenience along with profitability with reasonable ticket prices as ODMTS is designed under fixed pricing for existing riders and convenience concerned potential riders. 
}

\end{APPENDIX}

\bibliographystyle{informs2014} 
\bibliography{References} 

\end{document}